\documentclass[10pt,reqno]{amsart}
\usepackage{amsmath,amssymb,latexsym,esint,cite,mathrsfs}
\usepackage{verbatim,wasysym,cite,relsize,color}
\usepackage[latin1]{inputenc}
\usepackage{microtype}
\usepackage{color,enumitem,graphicx}
\usepackage[colorlinks=true,urlcolor=blue, citecolor=red,linkcolor=blue,
linktocpage,pdfpagelabels, bookmarksnumbered,bookmarksopen]{hyperref}
\usepackage[hyperpageref]{backref}
\usepackage[english]{babel}
\usepackage{tikz}
\usepackage[font=small,skip=5pt]{caption}
\usepackage{geometry}
\numberwithin{equation}{section}
\newtheorem{theorem}{Theorem}[section]
\newtheorem{lemma}[theorem]{Lemma}

\newtheorem{proposition}[theorem]{Proposition}
\newtheorem{corollary}[theorem]{Corollary}

\theoremstyle{definition}
\newtheorem{definition}[theorem]{Definition}
\newtheorem{remark}[theorem]{Remark}

\newcommand{\R}{\mathbb R}
\newcommand{\C}{\mathbb C}

\newcommand{\Kc}{\mathcal K}
\newcommand{\Hc}{\mathcal H}

\newcommand{\vareps}{\varepsilon}
\newcommand{\eps}{\epsilon}
\DeclareMathOperator*{\loc}{loc}
\DeclareMathOperator*{\rad}{rad}
\DeclareMathOperator*{\opt}{opt}

\DeclareMathOperator*{\betc}{{\beta_c}}
\DeclareMathOperator*{\gamc}{{\gamma_c}}
\DeclareMathOperator*{\sigc}{{\sigma_c}}
\DeclareMathOperator*{\deltc}{{\delta_c}}
\DeclareMathOperator*{\ima}{Im}
\DeclareMathOperator*{\rea}{Re}
\newcommand{\scal}[1]{\left\langle #1 \right\rangle}

\title[Global dynamics for INLS]
{Global dynamics for a class of inhomogeneous nonlinear Schr\"odinger equations with potential}
\author[V. D. Dinh]{Van Duong Dinh}
\address[V. D. Dinh]{Laboratoire Paul Painlev\'e UMR 8524, Universit\'e de Lille CNRS, 59655 Villeneuve d'Ascq Cedex, France
and 
Department of Mathematics, HCMC University of Education, 280 An Duong Vuong, Ho Chi Minh, Vietnam}
\email{contact@duongdinh.com}


\begin{document}
	
	\begin{abstract}
	We consider a class of $L^2$-supercritical inhomogeneous nonlinear Schr\"odinger equations with potential in three dimensions
	\[
	i\partial_t u + \Delta u - V u = \pm |x|^{-b} |u|^\alpha u, \quad (t,x) \in \R \times \R^3,
	\]
	where $0<b<1$ and $\alpha>\frac{4-2b}{3}$. In the focusing case, using a recent argument of Dodson-Murphy \cite{DM}, we first study the energy scattering below the ground state for the equation with radially symmetric initial data. We then establish blow-up criteria for the equation whose proof is based on an argument of Du-Wu-Zhang \cite{DWZ}. In the defocusing case, we also prove the energy scattering for the equation with radially symmetric initial data.
	\end{abstract}

	\maketitle

	\section{Introduction}
	\label{S1}
	\setcounter{equation}{0}
	
	\subsection{Introduction}
	We consider the Cauchy problem for the inhomogeneous nonlinear Schr\"odinger equations with potential
	\begin{equation} \label{INLS-V}
		\left\{ 
		\begin{array}{rcl}
			i\partial_t u + \Delta u - Vu &=& \pm |x|^{-b} |u|^\alpha u, \quad (t,x) \in \R \times \R^3, \\
			u(0)&=& u_0,
		\end{array}
		\right.
	\end{equation}
	where $u: \mathbb{R} \times \mathbb{R}^3 \rightarrow \mathbb{C}$, $u_0: \mathbb{R}^3 \rightarrow \mathbb{C}$, $b>0$, $\alpha>0$ and $V$ is a real-valued potential. The plus (resp. minus) sign in front of the nonlinearity corresponds to the defocusing (resp. focusing) case. The inhomogeneous nonlinear Schr\"odinger equation arises in nonlinear optics for the propagation of laser beams. The beam propagation can be modeled by the equation of the form
	\begin{align} \label{model-equ}
	i\partial_t u + \Delta u + K(x) |u|^\alpha u=0.
	\end{align}
	The equation \eqref{model-equ} has been attracted much attention recently. Berg\'e \cite{Berge} studied formally the stability condition for solition solutions of \eqref{model-equ}. Towers-Malomed \cite{TM} observed by means of variational approximation and direct simulations that a certain type of time-dependent nonlinear medium gives rise to completely stable beams. Merle \cite{Merle} and Rapha\"el-Szeftel \cite{RS} studied the existence and non-existence of minimal mass blow-up solutions  for \eqref{model-equ} with $k_1 < K(x) <k_2$ and $k_1, k_2>0$. Fibich-Wang \cite{FW} investigated the stability of solitary waves for \eqref{model-equ} with $K(x)= K(\eps |x|)$, where $\eps>0$ is a small parameter and $K \in C^4(\R^N) \cap L^\infty(\R^N)$. The case $K(x) = |x|^b$ with $b>0$ was studied in \cite{Chen, CG, LWW, Zhu}.
	
	In this paper, the potential $V: \mathbb{R}^3 \rightarrow \mathbb{R}$ is assumed to satisfy
	\begin{align} \label{ass-V-1}
	V \in \Kc_0 \cap L^{\frac{3}{2}}
	\end{align}
	and
	\begin{align} \label{ass-V-2}
	\|V_-\|_{\Kc} <4\pi,
	\end{align}
	where $\Kc_0$ is defined as the closure of bounded and compactly supported functions with respect to the Kato norm
	\[
	\|V\|_{\Kc} := \sup_{x \in \R^3} \int_{\R^3} \frac{|V(y)|}{|x-y|} dy,
	\]
	and $V_-:= \min \{V, 0\}$ is the negative part of $V$. Under the assumptions \eqref{ass-V-1} and \eqref{ass-V-2}, the operator $\Hc:= -\Delta +V$ has no eigenvalues, and by \cite{Hong}, the Schr\"odinger operator $e^{-it\Hc}$ enjoys dispersive and Strichartz estimates. Moreover, under the assumptions \eqref{ass-V-1} and \eqref{ass-V-2}, the Sobolev norms
	\begin{align} \label{grad-V}
	\|\Lambda f\|^2_{L^2}:= \int |\nabla f|^2 dx + \int V|f|^2 dx, \quad \|\nabla f\|^2_{L^2}, \quad f \in H^1
	\end{align}
	are equivalent.
	
	The main purpose of this sequel is to study global dynamics such as energy scattering and blow-up criteria for the nonlinear equation \eqref{INLS-V}. Thanks to Strichartz estimates, the Cauchy problem \eqref{INLS-V} is locally well-posed in $H^1$ (see Lemma $\ref{lem-lwp}$). Moreover, local solutions enjoy the conservation of mass and energy
	\begin{align*}
	M(u(t)) &= \int |u(t,x)|^2 dx = M(u_0), \tag{\text{Mass}} \\
	E(u(t)) &= \frac{1}{2} \int |\nabla u(t,x)|^2 dx +\frac{1}{2} \int V(x) |u(t,x)|^2 dx \pm \frac{1}{\alpha+2} \int |x|^{-b} |u(t,x)|^{\alpha+2} dx = E(u_0). \tag{\text{Energy}}
	\end{align*}
	
	\subsection{Known results}
	Before stating our results, let us recall known results for the inhomogeneous nonlinear Schr\"odinger equations without potential, namely
	\begin{equation} \label{INLS}
	\left\{ 
	\begin{array}{rcl}
	i\partial_t u + \Delta u &=& \pm |x|^{-b} |u|^\alpha u, \quad (t,x) \in \R \times \R^N, \\
	u(0)&=& u_0.
	\end{array}
	\right.
	\end{equation}
	The equation \eqref{INLS} is invariant under the scaling
	\[
	u_\lambda(t,x):= \lambda^{\frac{2-b}{\alpha}} u(\lambda^2 t, \lambda x), \quad \lambda>0.
	\]
	An easy computation shows 
	\[
	\|u_\lambda(0)\|_{\dot{H}^\gamma} = \lambda^{\gamma-\frac{N}{2}+\frac{2-b}{\alpha}} \|u_0\|_{\dot{H}^\gamma}.
	\]
	We thus denote the critical exponents
	\[
	\gamc := \frac{N}{2} - \frac{2-b}{\alpha}
	\]
	and 
	\begin{align} \label{def-sigc}
	\sigc:= \frac{1-\gamc}{\gamc} = \frac{4-2b-(N-2)\alpha}{N\alpha-4+2b}.
	\end{align}
	The equation \eqref{INLS} has formally the conservation of mass and energy
	\begin{align}
	M(u(t)) &= \int |u(t,x)|^2 dx = M(u_0),   \\
	E_0(u(t)) &= \frac{1}{2} \int |\nabla u(t,x)|^2 dx \pm \frac{1}{\alpha+2} \int |x|^{-b} |u(t,x)|^{\alpha+2} dx = E_0(u_0). 
	\end{align}
	The well-posedness for \eqref{INLS} with initial data in $H^1$ was first studied by Genoud-Stuart \cite{GS} by using an abstract theory of Cazenave \cite[Chapter 3]{Cazenave} which does not use Strichartz estimates. More precisely, they proved that the focusing problem \eqref{INLS} with $0<b<\min\{2,N\}$ is well posed in $H^1$:
	\begin{itemize}
		\item locally if $0<\alpha<2^*$,
		\item globally for any initial data if $0<\alpha<2_*$,
		\item globally for small initial data if $2_* \leq \alpha<2^*$,
	\end{itemize}
	where
	\begin{align} \label{def-2-star}
	2^*:= \left\{
	\begin{array}{cl}
	\frac{4-2b}{N-2} & \text{if } N\geq 3, \\
	\infty &\text{if } N=1,2,
	\end{array}
	\right.
	\quad
	2_*:= \frac{4-2b}{N}.	
	\end{align}
	
	Guzm\'an \cite{Guzman} and Dinh \cite{Dinh-weigh} later used Strichartz estimates and the contraction mapping argument to show the local well-posedness for \eqref{INLS}. They proved that if
	\begin{align*}
	\renewcommand*{\arraystretch}{1.3}
	\left\{
	\begin{array}{l}
	N\geq 4, \quad 0<b<2, \quad 0<\alpha<2^*, \\
	N=3, \quad 0<b<1, \quad 0<\alpha<2^*, \\
	N=3, \quad 1 \leq b <\frac{3}{2}, \quad 0<\alpha<\frac{6-4b}{2b-1}, \\
	N=2, \quad 0<b<1, \quad 0<\alpha<2^*,
	\end{array}
	\right.
	\end{align*}
	then \eqref{INLS} is locally well-posed in $H^1$. Moroever, the local solution satisfies $u \in L^q_{\loc}((-T_*, T^*), W^{1,r})$ for any Schr\"odinger admissible pair $(q,r)$, where $(-T_*, T^*)$ is the maximal time of existence. Note that the results of Guzm\'an and Dinh are weaker than the ones of Genoud-Stuart. It does not treat the case $N=1$ and there are restrictions on the validity of $b$ when $N=2$ and $N=3$. However, it shows that the solution belongs locally in Strichartz spaces $L^q((-T_*, T^*), W^{1,r})$.  This property plays a crucial role in the scattering theory. 
	
	In the case $\alpha=2_*$, Genoud \cite{Genoud} showed that the focusing problem \eqref{INLS} with $0<b<\min \{2,N\}$ is globally well-posed in $H^1$ by assuming $u_0 \in H^1$ and $\|u_0\|_{L^2} <\|Q\|_{L^2}$, where $Q$ is the unique positive radially symmetric and decreasing solution to the elliptic equation
	\begin{align} \label{ell-equ}
	\Delta Q-Q+|x|^{-b} |Q|^\alpha Q=0.
	\end{align}
	Combet-Genoud \cite{CG} later established the classification of minimal mass blow-up solutions to the focusing problem \eqref{INLS}. Note that the uniqueness of positive radial solution to \eqref{ell-equ} was established by Yanagida \cite{Yanagida} and Genoud \cite{Genoud-2d}. Their results hold under the assumptions $0<b<\min \{2,N\}$ and $0<\alpha<2^*$. 
	
	In the case $2_*<\alpha<2^*$, Farah \cite{Farah} proved that the focusing problem \eqref{INLS} with $0<b<\min \{2,N\}$ is globally well-posed in $H^1$ provided that $u_0 \in H^1$ and satisfies
	\begin{align} \label{con-energy}
	E_0(u_0) [M(u_0)]^{\sigc} < E_0(Q) [M(Q)]^{\sigc}
	\end{align}
	and
	\begin{align} \label{con-grad}
	\|\nabla u_0\|_{L^2} \|u_0\|_{L^2}^{\sigc} < \|\nabla Q\|_{L^2} \|Q\|^{\sigc}_{L^2},
	\end{align}
	where $\sigc$ is as in \eqref{def-sigc}. The existence of finite time blow-up solutions for the focusing problem \eqref{INLS} was studied by Farah \cite{Farah} and Dinh \cite{Dinh-blow}. 
	
	The energy scattering for the focusing problem \eqref{INLS} was first established by Farah-Guzm\'an \cite{FG-3D} with $0<b<\frac{1}{2}$, $\alpha=2$ and $N=3$. The proof is based on the concentration-compactness argument developed by Kenig-Merle \cite{KM}. This result was later extended to higher dimensions in \cite{FG-high} using again the concentration-compactness argument. Recently, Campos \cite{Campos} used a new method of Dodson-Murphy \cite{DM} to give an alternative simple proof for the results of Farah-Guzm\'an. He also extends the validity of $b$ in dimensions $N\geq 3$. In the case $N=2$, the energy scattering for the focusing problem \eqref{INLS} was first established with $0<b<\frac{2}{3}$ and $\alpha>2-b$ by Farah-Guzm\'an \cite{FG-high} via the concentration-compactness argument. Recently, Xu-Zhao \cite{XZ} and Dinh \cite{Dinh-2D} simultaneously proved the energy scattering for the focusing problem \eqref{INLS}  with $0<b<1$ and $\alpha>2-b$ by adapting a new argument of Arora-Dodson-Murphy \cite{ADM}.

	In the defocusing case, the energy scattering for \eqref{INLS} was first established in \cite{Dinh-weigh} by considering the initial data in the weighted $L^2$ space $\Sigma:= H^1 \cap L^2(|x|^2 dx)$. The energy scattering for the defocusing problem \eqref{INLS} with $H^1$ initial data in dimensions $N\geq 3$ was proved in \cite{Dinh-scat}. The proof is based on the decay property of global solutions. Recently, Dinh \cite{Dinh-2D} proved the energy scattering with radially symmetric initial data for the defocusing problem \eqref{INLS} in dimension $N=2$. An alternative simple proof of the energy scattering for the defocusing problem \eqref{INLS} with $H^1$ initial data (not necessary radially symmetric) in dimensions $N\geq 3$ was also given in \cite[Appendix]{Dinh-2D}. 
		
	We next recall some known results for the nonlinear Schr\"odinger equations with potential
	\begin{align} \label{NLS-V}
	\left\{
	\begin{array}{rcl}
	i\partial_t w + \Delta w - Vw &=& \pm |w|^\alpha w, \quad (t,x) \in \R \times \R^3, \\
	w(0)&=& w_0.
	\end{array}
	\right.
	\end{align}
	Under the assumptions \eqref{ass-V-1} and \eqref{ass-V-2}, the equation \eqref{NLS-V} is locally well-posed in $H^1$ (see e.g. \cite{Hong,HI}). Moreover, local solutions enjoy the conservation of mass and energy
	\begin{align*}
	M(w(t)) &= \int |w(t,x)|^2 dx = M(w_0),   \\
	E(w(t)) &= \frac{1}{2} \int |\nabla w(t,x)|^2 dx +\frac{1}{2} \int V(x)|w(t,x)|^2 dx \pm \frac{1}{\alpha+2} \int |w(t,x)|^{\alpha+2} dx = E(w_0). 
	\end{align*}
	The energy scattering for the focusing problem \eqref{NLS-V} was first studied by Hong \cite{Hong} with $\alpha=2$. More precisely, he proved the following result.
	\begin{theorem}[\cite{Hong}]
		Let $\alpha=2$ and $V: \R^3 \rightarrow \R$ satisfy \eqref{ass-V-1}, $V\geq 0$, $x\cdot \nabla V \leq 0$ and $x\cdot \nabla V \in L^{\frac{3}{2}}$. Let $w_0 \in H^1$ be such that
		\[
		E(w_0) M(w_0) < E_0(W) M(W), \quad \|\Lambda w_0\|_{L^2} \|w_0\|_{L^2} < \|\nabla W\|_{L^2} \|W\|_{L^2},
		\]
		where
		\[
		E_0(w):=\frac{1}{2} \int |\nabla w(x)|^2 dx - \frac{1}{\alpha+2} \int |w(x)|^{\alpha+2} dx
		\]
		and $W$ is the unique positive radial solution to 
		\[
		\Delta W - W + |W|^\alpha W=0.
		\]
		Then the corresponding solution to the focusing problem \eqref{NLS-V} exists globally in time and scatters in both directions, i.e. there exist $w_0^\pm \in H^1$ such that
		\[
		\lim_{t\rightarrow \pm \infty} \|w(t)- e^{-it\Hc} w_0^\pm \|_{H^1} =0.
		\]
	\end{theorem} 
	The proof of this result is again based on the concentration-compactness argument of Kenig-Merle. Recently, Hamano-Ikeda extended Hong's results to the whole range of the intercritical case, i.e. $\frac{4}{3}<\alpha<4$ and radially symmetric initial data. The proof makes use of the argument of Dodson-Murphy \cite{DM}. They also established blow-up criteria for the equation by applying the argument of Du-Wu-Zhang \cite{DWZ}. More precisely, their results read as follows.
	\begin{theorem}[\cite{HI}]
		Let $\frac{4}{3}<\alpha<4$. Let $V: \R^3 \rightarrow \R$ satisfy $V\geq 0$ and $x\cdot \nabla V \in L^{\frac{3}{2}}$. Let $w_0 \in H^1$ satisfy
		\[
		E(w_0) [M(w_0)]^{\deltc} < E_0(W) [M(W)]^{\deltc},
		\]
		where $\deltc:=\frac{1-\betc}{\betc} =\frac{4-\alpha}{3\alpha-4}$ with $\betc:=\frac{3}{2} - \frac{2}{\alpha}$. Then it holds that
		\begin{itemize}
			\item (Global existence and scattering) If $V \in \mathcal{K}_0 \cap L^{\frac{3}{2}}, x \cdot \nabla V \leq 0$ and 
			\[
			\|\nabla w_0\|_{L^2} \|w_0\|^{\deltc}_{L^2} <\|\nabla W\|_{L^2} \|W\|^{\deltc}_{L^2},
			\]
			then the corresponding solution exists globally in time and satisfies
			\[
			\|\nabla w(t)\|_{L^2} \|w(t)\|^{\deltc}_{L^2} < \|\nabla W\|_{L^2} \|W\|^{\deltc}_{L^2}
			\]
			for all $t\in \R$. Moreover, if $w_0$ and $V$ are radially symmetric, then the solution scatters in both directions.
			\item (Blow-up) If either $V \in \mathcal{K}_0 \cap L^{\frac{3}{2}}$ or $V \in L^\sigma$ for some $\sigma>\frac{3}{2}$, and assume $2V+ x \cdot \nabla V \geq 0$ and 
			\[
			\|\Lambda w_0\|_{L^2} \|w_0\|^{\deltc}_{L^2}>\|\nabla W\|_{L^2} \|W\|^{\deltc}_{L^2},
			\]
			then 
			\[
			\|\Lambda w(t)\|_{L^2} \|w(t)\|^{\deltc}_{L^2} > \|\nabla W\|_{L^2} \|W\|^{\deltc}_{L^2}
			\]
			for all $t \in (-T_*,T^*)$, where $(-T_*,T^*)$ is the maximal time interval of existence. Moreover, either $T^*<+\infty$, or $T^*=+\infty$ and there exists $t_n \rightarrow +\infty$ such that
			\[
			\lim_{n\rightarrow \infty} \|\nabla w(t_n)\|_{L^2} = \infty.
			\]
			A similar conclusion holds for $T_*$. Furthermore, if $x\cdot \nabla V \geq 0$ and either
			\begin{itemize}
				\item[(i)] $w_0$ and $V$ are radially symmetric and $V \in L^\sigma$ for some $\sigma>\frac{3}{2}$
			\end{itemize}
			or
			\begin{itemize}
				\item[(ii)] $w_0 \in L^2(|x|^2 dx)$ and either $V \in \mathcal{K}_0\cap L^{\frac{3}{2}}$ or $V\in L^\sigma$ for some $\sigma>\frac{3}{2}$,
			\end{itemize}
			then $T_*<+\infty$ and $T^*<+\infty$.
		\end{itemize}
	\end{theorem}
	
	Concerning the energy scattering for the defocusing problem \eqref{NLS-V}, the following result was proved in \cite[Appendix]{Hong}.
	
	\begin{theorem}[\cite{Hong}]
		Let $\alpha=2$. Let $V: \R^3 \rightarrow \R$ satisfy \eqref{ass-V-1}, \eqref{ass-V-2} and $\|(x\cdot \nabla V)_+\|_{\mathcal{K}} <4\pi$. Let $u_0 \in H^1$. Then the corresponding solution to the defocusing problem \eqref{NLS-V} exists globally in time and scatters in both directions.
	\end{theorem}
	Since we are not aware of any other energy scattering results for the defocusing problem \eqref{NLS-V}, we prove the following result whose proof is given in the Appendix.
	
	\begin{theorem} \label{theo-scat-defocus-NLS}
		Let $\frac{4}{3}<\alpha<4$. Let $V: \R^3 \rightarrow \R$ be radially symmetric satisfying \eqref{ass-V-1}, \eqref{ass-V-2}, $x\cdot \nabla V \leq 0$ and $\partial_r V \in L^q$ for any $\frac{3}{2}\leq q \leq \infty$. Let $u_0 \in H^1$. Then the corresponding solution to the defocusing problem \eqref{NLS-V} exists globally in time and scatters in both directions.
	\end{theorem}
	
	Note that the condition $\partial_r V \in L^q$ for any $\frac{3}{2}\leq q \leq \infty$ is needed to ensure $\partial_r V |u(t)|^2 \in L^1$ (see Remark \ref{rem-Lq}). This assumption can be relaxed to $\partial_r V \in L^q + L^\infty$ for some $q\geq \frac{3}{2}$.
	
	\subsection{Main results}
	In this paper, we extend the results of Hong \cite{Hong} and Hamano-Ikeda \cite{HI} to a class of mass-supercritical inhomogeneous nonlinear Schr\"odinger equations with potential. Our main result is the following.
	\begin{theorem} \label{theo-dyna-focus}
		Let $0<b<1$ and $\frac{4-2b}{3}<\alpha<4-2b$. Let $V: \R^3 \rightarrow \R$ satisfy \eqref{ass-V-1} and $V\geq 0$. Let $u_0 \in H^1$ satisfy 
		\begin{align} \label{cond-ener}
		E(u_0) [M(u_0)]^{\sigc} < E_0(Q) [M(Q)]^{\sigc}.
		\end{align}
		\begin{itemize}
			\item  (Global existence and scattering) If 
			\begin{align} \label{cond-grad-glob}
			\|\Lambda u_0\|_{L^2} \|u_0\|^{\sigc}_{L^2} < \|\nabla Q\|_{L^2} \|Q\|^{\sigc}_{L^2},
			\end{align}
			then the corresponding solution to the focusing problem \eqref{INLS-V} exists globally in time and satisfies
			\begin{align} \label{est-solu-glob}
			\|\Lambda u(t)\|_{L^2} \|u(t)\|^{\sigc}_{L^2} < \|\nabla Q\|_{L^2} \|Q\|^{\sigc}_{L^2}
			\end{align}
			for all $t\in \R$. Moreover, if $x\cdot \nabla V \in L^{\frac{3}{2}}, x \cdot \nabla V \leq 0$, $u_0, V$ are radially symmetric and $\frac{4-2b}{3}<\alpha<3-2b$, then the global solution scatters in both directions.
			\item (Blow-up) If
			\begin{align} \label{cond-grad-blow}
			\|\Lambda u_0\|_{L^2} \|u_0\|^{\sigc}_{L^2} > \|\nabla Q\|_{L^2} \|Q\|^{\sigc}_{L^2},
			\end{align}
			then the corresponding solution to the focusing problem \eqref{INLS-V} satisfies
			\[
			\|\Lambda u(t)\|_{L^2} \|u(t)\|^{\sigc}_{L^2}> \|\nabla Q\|_{L^2} \|Q\|^{\sigc}_{L^2}
			\]
			for all $t \in (-T_*,T^*)$, where $(-T_*,T^*)$ is the maximal time interval of existence. Moreover, if 
			\begin{align} \label{cond-blow}
			x \cdot \nabla V \in L^{\frac{3}{2}}, \quad 2V + x \cdot \nabla V \geq 0,
			\end{align}
			then either $T^*<+\infty$ or $T^* = +\infty$ and there exists a time sequence $t_n \rightarrow +\infty$ such that
			\[
			\lim_{n\rightarrow \infty} \|\nabla u(t_n)\|_{L^2} =\infty.
			\]
			A similar conclusion holds for $T_*$. In addition to \eqref{cond-blow}, if $u_0 \in L^2(|x|^2 dx)$, then $T_*<+\infty$ and $T^*<+\infty$.
		\end{itemize}
	\end{theorem}
	
	\begin{remark}
		We will see in Remark $\ref{rem-glob-refi}$ that the condition \eqref{cond-grad-glob} can be replaced by
		\begin{align} \label{cond-grad-glob-refi}
		\|\nabla u_0\|_{L^2} \|u_0\|^{\sigc}_{L^2} < \|\nabla Q\|_{L^2} \|Q\|^{\sigc}_{L^2}.
		\end{align}
		In this case, \eqref{est-solu-glob} becomes
		\[
		\|\nabla u(t)\|_{L^2} \|u(t)\|^{\sigc}_{L^2} < \|\nabla Q\|_{L^2} \|Q\|^{\sigc}_{L^2}
		\]
		for all $t\in \R$. 
	\end{remark}

	\begin{remark}
		\begin{itemize}
			\item There is a restriction $\frac{4-2b}{3}<\alpha<3-2b$ for the energy scattering. This restriction is due to the equivalence of Sobolev norms (see Remark $\ref{rem-non-est}$).
			\item It was noticed in \cite[Remark 1.5]{HI} that if $V$ is radial, $V\geq 0$ and $x\cdot \nabla V\geq 0$, then $V \notin L^{\frac{3}{2}}$. We thus do not get the finite time blow-up for radial initial data.
		\end{itemize}
	\end{remark}

	\begin{remark}
		Recently, Guo-Wang-Yao \cite{GWY} used the concentration-compactness argument to show the energy scattering for \eqref{INLS-V} with $\alpha=2$ and $0<b<1$. Comparing to their result, our result is weaker and only holds for $\alpha=2$ and $0<b<\frac{1}{2}$. The result in \cite{GWY} relies crucially on the following nonlinear estimates (see also Lemma $\ref{lem-non-est}$)
		\begin{align*}
		\||x|^{-b} |u|^2 u\|_{S'(\dot{H}^{-\gamc},I)} &\lesssim \||\nabla|^{\gamc} u\|_{L^\infty(I,L^2)} \|u\|_{L^{\frac{4}{1-b}}(I,L^6)} \|u\|_{L^\infty(I,L^{\frac{6}{2-b}})}, \\
		\||x|^{-b} |u|^2 u\|_{L^2(I,L^{\frac{6}{5}})} &\lesssim \||\nabla|^{\gamc} u\|_{L^4(I,L^3)} \|u\|_{L^{\frac{4}{b}}(I,L^{\frac{6}{3-b}})} \|u\|_{L^{\frac{4}{1-b}}(I,L^6)}, \\
		\|\nabla(|x|^{-b} |u|^2 u)\|_{L^2(I,L^{\frac{6}{5}})} &\lesssim \||\nabla|^{\gamc} u\|_{L^4(I,L^3)} \|\nabla u\|_{L^{\frac{4}{b}}(I,L^{\frac{6}{3-b}})} \|u\|_{L^{\frac{4}{1-b}}(I,L^6)}, 
		\end{align*}
		where $(\infty, 2), (4,3), \left(\frac{4}{b},\frac{6}{3-b}\right) \in S$ and $\left(\frac{4}{1-b},6 \right), \left(\infty,\frac{6}{2-b}\right) \in S_{\gamc}$ with $0<b<1$. The proof of these estimates is based on the weighted Sobolev estimates of Stein-Weiss \cite{SW} and the algebraic nature of the cubic nonlinearity. 
	\end{remark}

	We also have the following energy scattering for the defocusing problem \eqref{INLS-V} with radially symmetric initial data.
		
	\begin{theorem} \label{theo-scat-defocus}
		Let $0<b<1$ and $\frac{4-2b}{3}<\alpha<3-2b$. Let $V: \R^3 \rightarrow \R$ satisfy \eqref{ass-V-1}, \eqref{ass-V-2}, $x \cdot \nabla V \in L^{\frac{3}{2}}$, $x\cdot \nabla V\leq 0$ and $V$ be radially symmetric. Let $u_0 \in H^1$ be radially symmetric. Then the corresponding solution to the defocusing problem \eqref{INLS-V} exists globally in time and scatters in both directions.
	\end{theorem}
	
	\subsection{Idea of the proof}
	The proofs of the energy scattering for radially symmetric initial data given in Theorem $\ref{theo-dyna-focus}$ and Theorem $\ref{theo-scat-defocus}$ are based on the argument of Dodson-Murphy \cite{DM}. The first step is to use the variational argument to derive the coercivity on sufficiently large balls, that is, there exist $\delta>0$ and $R_0>0$ (depending on $u_0, Q$ in the focusing case and on $u_0$ in the defocusing case) such that for any $R\geq R_0$,
	\[
	H(\chi_R u(t)) \geq \delta \int |x|^{-b} |\chi_R u(t,x)|^{\alpha+2} dx
	\]
	for all $t\in \R$, where $H$ is the virial functional defined by
	\[
	H(u):= \|\nabla u\|^2_{L^2} -\frac{3\alpha+2b}{2(\alpha+2)} \int |x|^{-b} |u(x)|^{\alpha+2} dx
	\]
	and $\chi_R(x) = \chi(x/R)$ with $\chi\in C^\infty_0(\R^3)$ satisfying $0\leq \chi\leq 1$, $\chi(x) =1$ on $|x| \leq 1/2$ and $\chi(x)=0$ on $|x|\geq 1$. Using this coercivity, the Morawetz estimate and the radial Sobolev embedding imply that for any $T>0$ and any $R\geq R_0$, the solution satisfies the space-time estimate
	\[
	\frac{1}{T} \int_0^T \int_{|x| \leq R/2} |u(t,x)|^{\alpha+2+b} dx dt \lesssim \frac{R}{T} +\frac{1}{R^2} + \frac{1}{R^{\alpha+b}} + o_R(1).
	\]
	By choosing a suitable time $T$, the fundamental theorem of calculus ensures the existence of a time sequence $t_n \rightarrow +\infty$ such that for any $R>0$,
	\begin{align} \label{small-L2-intro}
	\lim_{n\rightarrow \infty} \int_{|x| \leq R} |u(t_n,x)|^2 dx =0.
	\end{align}
	The second step is to show a suitable small data scattering by using nonlinear estimates related to the equation. More precisely, we prove that there exists $\varrho>0$ (depending on $u_0,Q$ in the focusing case and on $u_0$ in the defocusing case) such that if 
	\begin{align} \label{small-data-scat-intro}
	\|e^{-i(t-T)\Hc} u(T)\|_{S(\dot{H}^{\gamc},[T,+\infty))} <\varrho
	\end{align}
	for some $T>0$, then the solution scatters in $H^1$ forward in time, where $S(\dot{H}^{\gamc},I)$ is the $\dot{H}^{\gamc}$-admissible Strichartz space. The remain step is to show \eqref{small-data-scat-intro} for some $T>0$ sufficiently large. To this end, we use the Duhamel formula to write for any $t>T$,
	\[
	e^{-i(t-T)\Hc} u(T) = e^{it\Delta} u_0 + F_1(t)+F_2(t),
	\]
	where
	\[
	F_1(t):= i \int_{T-\vareps^{-\sigma}}^T e^{-i(t-s)\Hc} |x|^{-b} |u|^\alpha u(s) ds, \quad F_2(t):= i\int_0^{T-\vareps^{-\sigma}} e^{-i(t-s)\Hc} |x|^{-b} |u|^\alpha u(s) ds.
	\]
	The smallness of $\|e^{it\Delta} u_0\|_{S(\dot{H}^{\gamc},[T,+\infty))}$ follows easily from Strichartz estimates by taking $T>0$ sufficiently large. The smallness of $\|F_1\|_{S(\dot{H}^{\gamc},[T,+\infty))}$ follows from Strichartz estimates, \eqref{small-L2-intro} and the radial Sobolev embedding, while the smallness of $\|F_2\|_{(\dot{H}^{\gamc},[T,+\infty))}$ follows from the dispersive estimates. We refer the reader to Section $\ref{S3}$ for more details.
	
	The blow-up part given in Theorem $\ref{theo-dyna-focus}$ is based on the argument of Du-Wu-Zhang \cite{DWZ}. More precisely, if $u:[0,T^*) \times \R^3 \rightarrow \C$ is a $H^1$ solution to the focusing \eqref{INLS-V} satisfying
	\begin{align} \label{blowup-cond-intro}
	\sup_{t\in[0,T^*)} K(u(t)) \leq -\delta
	\end{align}
	for some $\delta>0$, where
	\[
	K(u(t)):= \|\nabla u(t)\|^2_{L^2} - \frac{1}{2} \int x \cdot  \nabla V |u(t)|^2 dx - \frac{3\alpha+2b}{2(\alpha+2)} \int |x|^{-b} |u(t)|^{\alpha+2} dx,
	\]
	then either $T^*<+\infty$ or $T^*=+\infty$ and there exists a time sequence $t_n\rightarrow +\infty$ such that $\lim_{n\rightarrow \infty} \|\nabla u(t_n)\|_{L^2}=\infty$. The proof of \eqref{blowup-cond-intro} is again a consequence of Morawetz-type estimates. We refer the reader to Section $\ref{S4}$ for more details.
	
	This paper is organized as follows. In Section $\ref{S2}$, we  give some preliminaries including Strichartz estimates, the equivalence of Sobolev norms and the local well-posedness. In Section $\ref{S3}$, we give the proofs of the energy scattering given in Theorem $\ref{theo-dyna-focus}$ and Theorem $\ref{theo-scat-defocus}$. In Section $\ref{S4}$, we study the blow-up for the focusing problem \eqref{INLS-V}.  Finally, we prove the energy scattering for the defocusing  problem \eqref{NLS-V} given in Theorem $\ref{theo-scat-defocus-NLS}$ in the Appendix.
	
	\section{Preliminaries}
	\label{S2}
	\setcounter{equation}{0}
	\subsection{Strichartz estimates}
	Let $I\subset \R$ and $q,r \in [1,\infty]$. We define the mixed norm 
	\[
	\|u\|_{L^q(I,L^r)} := \left( \int_I \left( \int_{\R^3} |u(t,x)|^r dx \right)^{\frac{q}{r}} dt \right)^{\frac{1}{q}}
	\]
	with a usual modification when either $q$ or $r$ are infinity. When $q=r$, we use the notation $L^q(I \times \R^3)$ instead of $L^q(I,L^q)$.
	\begin{definition}
		A pair $(q,r)$ is said to be Schr\"odinger admissible, for short $(q,r) \in S_0$, if
		\[
		(q,r) \in [2,\infty]^2, \quad (q,r) \ne (2,\infty), \quad \frac{2}{q}+\frac{3}{r} =\frac{3}{2}.
		\]
	\end{definition}
		
	\begin{lemma} [Dispersive estimate \cite{Hong}] 
		Let $V:\R^3 \rightarrow \R$ satisfy \eqref{ass-V-1} and \eqref{ass-V-2}. Then it holds that
		\begin{align} \label{disper-est}
		\|e^{-it\Hc} \|_{L^1 \rightarrow L^\infty} \lesssim |t|^{-\frac{3}{2}}.
		\end{align}
	\end{lemma}
	Thanks to this dispersive estimate and the abstract theory of Keel-Tao \cite{KT} (see also Foschi \cite{Foschi}), we have the following Strichartz estimates.
	
	\begin{proposition} [Strichartz estimates \cite{Hong}]
		Let $V:\R^3 \rightarrow \R$ satisfy \eqref{ass-V-1} and \eqref{ass-V-2}. Then it holds that
		\begin{align*}
		\|e^{-it\Hc} f\|_{L^q(\R, L^r)} &\lesssim \|f\|_{L^2}, \\
		\left\| \int_0^t e^{-i(t-s)\Hc} F(s) ds \right\|_{L^q(\R, L^r)} &\lesssim \|F\|_{L^{m'}(\R,L^{n'})},
		\end{align*}
		for any $(q,r), (m,n) \in S$, where $(m,m')$ and $(n,n')$ are H\"older's conjugate pairs.
	\end{proposition}
	
	\subsection{Equivalence of Sobolev norms}
	We define the homogeneous and inhomogeneous Sobolev spaces associated to $\Hc$ as the closure of $C^\infty_0(\R^3)$ under the norms
	\[
	\|f\|_{\dot{W}^{\gamma,r}_V} := \| \Lambda^\gamma f\|_{L^r}, \quad \|f\|_{W^{\gamma,r}_V} := \|\scal{\Lambda}^\gamma f\|_{L^r}, \quad \Lambda:= \sqrt{\Hc}.
	\]
	To simplify the notation, we denote $\dot{H}^\gamma_V:= \dot{W}^{\gamma,2}_V$ and $H^\gamma_V:= W^{\gamma,2}_V$.
	\begin{lemma} [Sobolev inequalities \cite{Hong}]
		Let $V: \R^3 \rightarrow \R$ satisfy \eqref{ass-V-1} and \eqref{ass-V-2}. Then it holds that
		\[
		\|f\|_{L^q} \lesssim \|f\|_{\dot{W}^{\gamma,p}_V}, \quad \|f\|_{L^q} \lesssim \|f\|_{W^{\gamma,p}_V},
		\]
		where $1<p<q<\infty$, $1<p<\frac{3}{\gamma}$, $0\leq \gamma \leq 2$ and $\frac{1}{q} = \frac{1}{p}-\frac{\gamma}{3}$.
	\end{lemma}
	
	\begin{lemma}[Equivalence of Sobolev spaces \cite{Hong}] \label{lem-equi-sobo}
		Let $V: \R^3 \rightarrow \R$ satisfy \eqref{ass-V-1} and \eqref{ass-V-2}. Then it holds that
		\[
		\|f\|_{\dot{W}^{\gamma,r}_V} \sim \|f\|_{\dot{W}^{\gamma,r}}, \quad \|f\|_{W^{\gamma,r}_V} \sim \|f\|_{W^{\gamma,r}},
		\]
		where $1<r<\frac{3}{\gamma}$ and $0\leq \gamma \leq 2$.
	\end{lemma}
	
	\subsection{Local well-posedness in $H^1$}
	In this subsection, we will show that under the assumptions \eqref{ass-V-1} and \eqref{ass-V-2}, the equation \eqref{INLS-V} is locally well-posed in $H^1$. To this end, we denote for any interval $I\subset \R$ the Strichartz norm
	\begin{align} \label{str-cha-norm}
	\|u\|_{S(L^2,I)} :=\sup_{\substack{(q,r) \in S_0 \\ 2 \leq r <3}} \|u\|_{L^q(I, L^r)}, \quad \|v\|_{S'(L^2,I)} := \inf_{\substack{(q,r)\in S_0 \\ 2 \leq r<3}} \|v\|_{L^{q'}(I,L^{r'})},
	\end{align}
	where $(q,q')$ and $(r,r')$ are H\"older's conjugate pairs. Here the condition $2 \leq r <3$ ensures $\dot{W}^{1,r}_V \sim \dot{W}^{1,r}$ and $\dot{W}^{1,r'}_V \sim \dot{W}^{1,r'}$.
	
	We also have the following nonlinear estimates.
	\begin{lemma} \label{lem-non-est-1}
		Let $0<b<1$, $0<\alpha<4-2b$ and $I \subset \R$. Then there exist positive numbers $\theta_1$ and $\theta_2$ such that 
		\begin{align*}
		\||x|^{-b} |u|^\alpha u\|_{S'(L^2,I)} &\lesssim \left(|I|^{\theta_1} + |I|^{\theta_2} \right) \|\nabla u\|^\alpha_{S(L^2,I)} \|u\|_{S(L^2,I)}, \\
		\|\nabla (|x|^{-b} |u|^\alpha u)\|_{S'(L^2,I)} &\lesssim \left(|I|^{\theta_1} + |I|^{\theta_2} \right) \|\nabla u\|^{\alpha+1}_{S(L^2,I)}.
		\end{align*}
	\end{lemma}
	
	\begin{proof}
		We only prove the second estimate, the first one is similar. 
		\[
		\|\nabla (|x|^{-b} |u|^\alpha u)\|_{S'(L^2,I)}  \lesssim \||x|^{-b} \nabla(|u|^\alpha u)\|_{S'(L^2,I)} + \||x|^{-b-1} |u|^\alpha u\|_{S'(L^2,I)},
		\]
		where we have used the fact $|\nabla(|x|^{-b})| =C(\gamma) |x|^{-b-1}$. We first estimate
		\begin{align*}
		\||x|^{-b} \nabla (|u|^\alpha u) \|_{S'(L^2,I)} \leq \||x|^{-b} \nabla (|u|^\alpha u) \|_{L^{q_1'}(I,L^{r_1'}(B))} +\||x|^{-b} \nabla (|u|^\alpha u) \|_{L^{q_2'}(I,L^{r_2'}(B^c))}
		\end{align*}
		for some $(q_1,r_1), (q_2,r_2) \in S_0$ satisfying $2 \leq r_1, r_2 <3$ to be chosen later, where $B:=B(0,1)$ and $B^c:= \R^3 \backslash B(0,1)$. By H\"older's inequality,
		\begin{align}
		\||x|^{-b} \nabla (|u|^\alpha u) \|_{L^{q_1'}(I,L^{r_1'}(B))}  &\leq \||x|^{-b}\|_{L^{\nu_1}(B)} \|\nabla (|u|^\alpha u)\|_{L^{q_1'}(I,L^{\rho_1})} \nonumber\\
		&\lesssim \|u\|^\alpha_{L^{m_1}(I,L^{n_1})} \|\nabla u\|_{L^{q_1}(I,L^{r_1})} \nonumber \\
		&\lesssim |I|^{\theta_1} \|\nabla u\|^{\alpha+1}_{L^{q_1}(I,L^{r_1})}  \label{non-est-1-proo-1}
		\end{align}
		provided $\nu_1, \rho_1, m_1, n_1 \geq 1$ satisfying
		\[
		\frac{1}{r_1'} = \frac{1}{\nu_1} +\frac{1}{\rho_1}, \quad \frac{3}{\nu_1}>b, \quad \frac{1}{\rho_1}=\frac{\alpha}{n_1} +\frac{1}{r_1}, \quad \frac{1}{q_1'} =\frac{\alpha}{m_1} +\frac{1}{q_1}
		\]
		and 
		\[
		\theta_1= \frac{\alpha}{m_1} -\frac{\alpha}{q_1}, \quad \frac{1}{n_1} =\frac{1}{r_1} -\frac{1}{3}.
		\]
		It follows that
		\[
		\frac{3}{\nu_1}=3-\frac{3(\alpha+2)}{r_1} + \alpha >b \quad \text{or} \quad  r_1 >\frac{3(\alpha+2)}{3+\alpha-b}.
		\]
		Let us choose $r_1=\frac{3(\alpha+2)}{3+\alpha-b} + \epsilon$ for some $0<\epsilon \ll 1$ to be chosen later. By taking $\epsilon>0$ small enough, we see that $2 <r<3$ since $0<b<1$. It remains to check $\theta_1>0$. This condition is equivalent to
		\[
		\frac{\alpha}{m_1} -\frac{\alpha}{q_1} = 1- \frac{\alpha+2}{q_1} >0 \quad \text{or} \quad  q_1>\alpha+2. 
		\]
		Since $(q_1,r_1) \in S_0$, we see that
		\[
		\frac{3}{2} -\frac{3}{r_1} = \frac{2}{q_1} <\frac{2}{\alpha+2}
		\]
		or
		\[
		3(\alpha+2) (4-2b - \alpha) + \epsilon(3+\alpha-b)(4-3(\alpha+2)>0.
		\]
		Since $0<\alpha<4-2b$, the above inequality holds by taking $\epsilon>0$ small enough. This shows that \eqref{non-est-1-proo-1} holds with some $\theta_1>0$, $(q_1,r_1) \in S_0$ and $2<r_1<3$. 
		
		On $B^c$, we simply take 
		\[
		q_2 = \frac{4(\alpha+2)}{\alpha}, \quad r_2 = \frac{3(\alpha+2)}{3+\alpha}. 
		\]
		Note that we have $2<r_2<3$. Let $m_2, n_2$ be such that
		\[
		\frac{1}{q_2'} =\frac{\alpha}{m_2} +\frac{1}{q_2}, \quad \frac{1}{r_2'} =\frac{\alpha}{n_2} +\frac{1}{r_2}.
		\]
		It is easy to check that
		\[
		\theta_2:= \frac{\alpha}{m_2}-\frac{\alpha}{q_2} = 1-\frac{\alpha+2}{q_2}=1-\frac{\alpha}{4}>0
		\]
		since $0<\alpha<4-2b$. We also have $\frac{1}{n_2}=\frac{1}{r_2}-\frac{1}{3}$ which implies that $\dot{W}^{1,r_2} \subset L^{n_2}$. With these choices, we have
		\begin{align}
		\||x|^{-b} \nabla (|u|^\alpha u) \|_{L^{q_2'}(I,L^{r_2'}(B^c))}  &\leq \||x|^{-b}\|_{L^\infty(B^c)} \|\nabla (|u|^\alpha u)\|_{L^{q_2'}(I,L^{r_2'})} \nonumber \\
		&\lesssim \|u\|^\alpha_{L^{m_2}(I,L^{n_2})} \|\nabla u\|_{L^{q_2}(I,L^{r_2})} \nonumber \\
		&\lesssim |I|^{\theta_2} \|\nabla u\|^{\alpha+1}_{L^{q_2}(I,L^{r_2})}. \label{non-est-1-proo-2}
		\end{align}
		We next estimate
		\[
		\||x|^{-b-1} |u|^\alpha u\|_{S'(L^2,I)} \leq \||x|^{-b-1} |u|^\alpha u\|_{L^{q_1'}(I,L^{r_1'}(B))} + \||x|^{-b-1} |u|^\alpha u\|_{L^{q_2'}(I,L^{r_2'}(B^c))}
		\]
		for some $(q_1,r_1), (q_2,r_2)\in S_0$ to be chosen shortly. By H\"older's inequality,
		\begin{align}
		\||x|^{-b-1} |u|^\alpha u\|_{L^{q_1'}(I,L^{r_1'}(B))} &\leq \||x|^{-b-1}\|_{L^{\nu_1}(B)} \||u|^\alpha u\|_{L^{q_1'}(I, L^{\rho_1})} \nonumber \\
		&\lesssim \|u\|^\alpha_{L^{m_1}(I,L^{n_1})} \|u\|_{L^{q_1}(I, L^{k_1})} \nonumber \\
		&\lesssim |I|^{\theta_1} \|\nabla u\|^{\alpha+1}_{L^{q_1}(I, L^{r_1})} \label{non-est-1-proo-3}
		\end{align}
		provided $\nu_1, \rho_1, m_1, n_1, k_1 \geq 1$ satisfying
		\[
		\frac{1}{r_1'} =\frac{1}{\nu_1}+\frac{1}{\rho_1}, \quad \frac{3}{\nu_1} >b+1, \quad \frac{1}{\rho_1} =\frac{\alpha}{n_1}+\frac{1}{k_1}, \quad \frac{1}{q_1'} =\frac{\alpha}{m_1} +\frac{1}{q_1}
		\]
		and
		\[
		\theta_1 = \frac{\alpha}{m_1} -\frac{\alpha}{q_1}, \quad \frac{1}{n_1}=\frac{1}{r_1}-\frac{1}{3}, \quad \frac{1}{k_1}=\frac{1}{r_1} -\frac{1}{3}.
		\]
		It follows that
		\[
		\frac{3}{\nu_1}= 3-\frac{3(\alpha+2)}{r_1} +\alpha+1 > b+1 \quad \text{or} \quad r_1>\frac{3(\alpha+2)}{3+\alpha-b}
		\]
		and
		\[
		\theta_1=\frac{\alpha}{m_1}-\frac{\alpha}{q_1} = 1-\frac{\alpha+2}{q_1} >0.
		\]
		These conditions are the sames as above, we thus can choose $q_1$ as in the first term to get \eqref{non-est-1-proo-3}. On $B^c$, we take $q_2, r_2, m_2, n_2$ as above and estimate
		\begin{align}
		\||x|^{-b-1} |u|^\alpha u\|_{L^{q_2'}(I,L^{r_2'}(B^c))} &\leq \||x|^{-b-1}\|_{L^3(B^c)} \|u\|^\alpha_{L^{m_2}(I, L^{n_2})} \|u\|_{L^{q_2}(I,L^{n_2})} \nonumber \\
		&\lesssim |I|^{\theta_2} \|\nabla u\|^{\alpha+1}_{L^{q_2}(I,L^{r_2})}, \label{non-est-1-proo-4}
		\end{align}
		where we have used the fact $0<b<1$ and
		\[
		\frac{1}{r_2'} =\frac{\alpha+1}{n_2} +\frac{1}{3}.
		\]
		Collecting \eqref{non-est-1-proo-1}--\eqref{non-est-1-proo-4}, we complete the proof.
	\end{proof}

	\begin{lemma} [Local well-posedness] \label{lem-lwp}
		Let $0<b<1$ and $0<\alpha<4-2b$. Let $V: \R^3 \rightarrow \R$ satisfy \eqref{ass-V-1} and \eqref{ass-V-2}. Then the equation \eqref{INLS-V} is locally well-posed in $H^1$.
	\end{lemma}
	
	\begin{proof}
		Consider 
		\[
		X:= \{ u \ : \ \|\scal{\Lambda} u\|_{S(L^2,I)} \leq M \}
		\]
		equipped with the distance
		\[
		d(u,v):= \|u-v\|_{S(L^2,I)},
		\]
		where $I=[0,T]$ with $T, M>0$ to be chosen later. We will show that the functional
		\[
		\Phi(u(t)) := e^{-it\Hc} u_0 \mp i \int_0^t e^{-i(t-s)\Hc} |x|^{-b}|u|^\alpha u (s) ds
		\]
		is a contraction on $(X,d)$. Thanks to Lemma $\ref{lem-non-est-1}$, we have
		\begin{align*}
		\|\scal{\Lambda} \Phi(u)\|_{S(L^2,I)} &\leq \|e^{-it\Hc} \scal{\Lambda} u_0\|_{S(L^2,I)} + \left\| \int_0^t e^{-i(t-s)\Hc} \scal{\Lambda}(|x|^{-b} |u|^\alpha u)(s) ds \right\|_{S(L^2,I)} \\
		&\leq C \|\scal{\Lambda} u_0\|_{L^2} + C \|\scal{\Lambda} (|x|^{-b} |u|^\alpha u)\|_{S'(L^2,I)} \\
		&\sim C \|u_0\|_{H^1} + C \|\scal{\nabla} (|x|^{-b} |u|^\alpha u) \|_{S'(L^2,I)} \\
		&\leq C \|u_0\|_{H^1} + C \left( |I|^{\theta_1} + |I|^{\theta_2} \right) \|\scal{\nabla} u\|^{\alpha+1}_{S(L^2,I)} \\
		&\sim C \|u_0\|_{H^1} + C \left( |I|^{\theta_1} + |I|^{\theta_2} \right) \|\scal{\Lambda} u\|^{\alpha+1}_{S(L^2,I)}
		\end{align*}
		for some $\theta_1, \theta_2>0$. Similarly, 
		\begin{align*}
		\|\Phi(u) - \Phi(v)\|_{S(L^2,I)} &\leq \left\| \int_0^t e^{-i(t-s) \Hc} |x|^{-b}(|u|^\alpha u - |v|^\alpha v)(s) ds\right\|_{S(L^2,I)} \\
		&\leq C \||x|^{-b}(|u|^\alpha u - |v|^\alpha v)\|_{S'(L^2,I)} \\
		&\leq C \left(|I|^{\theta_1} + |I|^{\theta_2} \right) \left(\|\scal{\nabla} u\|^\alpha_{S(L^2,I)} + \|\scal{\nabla} v\|^\alpha_{S(L^2,I)} \right) \|u-v\|_{S(L^2,I)} \\
		&\sim C \left(|I|^{\theta_1} + |I|^{\theta_2} \right) \left(\|\scal{\Lambda} u\|^\alpha_{S(L^2,I)} + \|\scal{\Lambda} v\|^\alpha_{S(L^2,I)} \right) \|u-v\|_{S(L^2,I)}.
		\end{align*}
		This shows that for any $u,v \in X$, there exists $C>0$ independent of $u_0$ and $T$ such that
		\begin{align*}
		\|\scal{\Lambda} \Phi(u)\|_{S(L^2,I)} &\leq C \|u_0\|_{H^1} + C (T^{\theta_1} + T^{\theta_2}) M^{\alpha+1}, \\
		d(\Phi(u), \Phi(v)) &\leq C (T^{\theta_1} + T^{\theta_2}) M^{\alpha} d(u,v).
		\end{align*}
		By choosing $M=2C \|u_0\|_{H^1}$ and taking $T>0$ sufficiently small so that
		\[
		C (T^{\theta_1} + T^{\theta_2}) M^{\alpha} <\frac{1}{2},
		\]
		the functional $\Phi$ is a contraction on $(X,d)$. The proof is complete.
	\end{proof}

\section{Energy scattering}
\label{S3}
\setcounter{equation}{0}
In this section, we give the proof of the energy scattering for the equation \eqref{INLS-V} given in Theorem $\ref{theo-dyna-focus}$ and Theorem $\ref{theo-scat-defocus}$. 

\subsection{Variational analysis}
Let us recall some properties related to the ground state $Q$  which is the unique positive radial decreasing solution to
\[
\Delta Q - Q + |x|^{-b} |Q|^\alpha Q =0.
\]
The ground state $Q$ optimizes the following Gagliardo-Nirenberg inequality: $0<b<2$ and $0<\alpha <4-2b$,
\[
\int |x|^{-b} |f(x)|^{\alpha+2} dx \leq C_{\opt} \|f\|^{\frac{4-2b - \alpha}{2}}_{L^2} \|\nabla f\|^{\frac{3\alpha+2b}{2}}_{L^2}, \quad f \in H^1(\R^3),
\]
that is,
\[
C_{\opt} = \int |x|^{-b} |Q(x)|^{\alpha+2} dx \div \left[\|Q\|^{\frac{4-2b - \alpha}{2}}_{L^2} \|\nabla Q\|^{\frac{3\alpha+2b}{2}}_{L^2} \right].
\]
It was shown in \cite{Farah} that $Q$ satisfies the following Pohozaev's identities
\[
\|Q\|^2_{L^2} = \frac{4-2b-\alpha}{3\alpha+2b} \|\nabla Q\|^2_{L^2} = \frac{4-2b-\alpha}{2(\alpha+2)} \int |x|^{-b} |Q(x)|^{\alpha+2} dx.
\]
In particular,
\begin{align} \label{GN-const}
C_{\opt} = \frac{2(\alpha+2)}{3\alpha+2b} \left( \|\nabla Q\|_{L^2} \|Q\|_{L^2}^{\sigc} \right)^{-\frac{3\alpha-4+2b}{2}},
\end{align}
where $\sigc$ is defined in \eqref{def-sigc}.

\begin{lemma} \label{lem-coer-1}
	Let $0<b<1$ and $\frac{4-2b}{3}<\alpha <4-2b$. Let $V :\R^3 \rightarrow \R$ satisfy \eqref{ass-V-1} and $V\geq 0$. Let $u_0 \in H^1$ satisfy \eqref{cond-ener}.
	\begin{itemize}
	\item If $u_0$ satisfies \eqref{cond-grad-glob}, then the corresponding solution to the focusing problem \eqref{INLS-V} satisfies
	\begin{align} \label{est-solu-1}
	\|\Lambda u(t)\|_{L^2} \|u(t)\|_{L^2}^{\sigc} < \|\nabla Q\|_{L^2} \|Q\|_{L^2}^{\sigc}
	\end{align}
	for all $t$ in the existence time. In particular, the corresponding solution to the focusing problem \eqref{INLS-V} exists globally in time. Moreover, there exists $\rho=\rho(u_0,Q)>0$ such that
	\begin{align} \label{coer-1}
	\|\Lambda u(t)\|_{L^2} \|u(t)\|_{L^2}^{\sigc} < (1-2\rho) \|\nabla Q\|_{L^2} \|Q\|_{L^2}^{\sigc}
	\end{align}
	for all $t\in \R$.
	\item If $u_0$ satisfies \eqref{cond-grad-blow}, then the corresponding solution to the focusing problem \eqref{INLS-V} satisfies
	\begin{align} \label{est-solu-blow}
	\|\Lambda u(t)\|_{L^2} \|u(t)\|^{\sigc}_{L^2} > \|\nabla Q\|_{L^2} \|Q\|^{\sigc}_{L^2}
	\end{align}
	for all $t$ in the existence time.
	\end{itemize} 
\end{lemma}

\begin{proof}
	We only prove the first item, the second one is similar. Multiplying both sides of $E(u(t))$ with $[M(u(t))]^{\sigc}$ and using the Gagliardo-Nirenberg inequality together with $V\geq 0$, we have
	\begin{align}
	E(u(t)) [M(u(t))]^{\sigc} &\geq \frac{1}{2} \left( \|\Lambda u(t)\|_{L^2} \|u(t)\|^{\sigc}_{L^2}\right)^2 -\frac{1}{\alpha+2} \left( \int |x|^{-b} |u(t,x)|^{\alpha+2} dx\right) \|u(t)\|^{2 {\sigc}}_{L^2} \nonumber \\
	& \geq \frac{1}{2} \left( \|\Lambda u(t)\|_{L^2} \|u(t)\|^{\sigc}_{L^2}\right)^2 -\frac{C_{\opt}}{\alpha+2} \|\nabla u(t)\|^{\frac{3\alpha+2b}{2}}_{L^2} \|u(t)\|^{\frac{4-2b-\alpha}{2} +2 {\sigc}}_{L^2} \nonumber \\
	& \geq \frac{1}{2} \left( \|\Lambda u(t)\|_{L^2} \|u(t)\|^{\sigc}_{L^2}\right)^2 -\frac{C_{\opt}}{\alpha+2} \|\Lambda u(t)\|^{\frac{3\alpha+2b}{2}}_{L^2} \|u(t)\|^{\frac{4-2b-\alpha}{2} +2 {\sigc}}_{L^2} \nonumber \\
	&= f\left(\|\Lambda u(t)\|_{L^2} \|u(t)\|^{\sigc}_{L^2}\right), \label{est-f}
	\end{align}
	where
	\[
	f(x):= \frac{1}{2} x^2 - \frac{C_{\opt}}{\alpha+2} x^{\frac{3\alpha+2b}{2}}.
	\]
	By Pohozaev's identities, a direction computation shows
	\begin{align} \label{energy-Q}
	f\left(\|\nabla Q\|_{L^2} \|Q\|^{\sigc}_{L^2}\right) = \frac{3\alpha-4+2b}{2(3\alpha+2b)} \left(\|\nabla Q\|_{L^2} \|Q\|^{\sigc}_{L^2}\right)^2 = E_0(Q) [M(Q)]^{\sigc}.
	\end{align}
	By \eqref{cond-ener}, the conservation of mass and energy, \eqref{est-f} and \eqref{energy-Q}, we infer that
	\[
	f\left(\|\Lambda u(t)\|_{L^2} \|u(t)\|^{\sigc}_{L^2}\right) < f\left(\|\nabla Q\|_{L^2} \|Q\|^{\sigc}_{L^2}\right)
	\]
	for all $t$ in the existence time. By \eqref{cond-grad-glob}, the continuity argument shows \eqref{est-solu-1}. Thus, by the conservation of mass and the local well-posedness, the corresponding solution exists globally in time. To see \eqref{coer-1}, we take $\theta = \theta(u_0,Q)>0$ such that
	\[
	E(u_0) [M(u_0)]^{\sigc} < (1-\theta) E_0(Q) [M(Q)]^{\sigc}.
	\]
	Using the fact
	\[
	E_0(Q) [M(Q)]^{\sigc} = \frac{3\alpha-4+2b}{2(3\alpha+2b)} \left(\|\nabla Q\|_{L^2} \|Q\|^{\sigc}_{L^2}\right)^2 = \frac{3\alpha-4+2b}{4(\alpha+2)} C_{\opt} \left(\|\nabla Q\|_{L^2} \|Q\|^{\sigc}_{L^2} \right)^{\frac{3\alpha+2b}{2}},
	\]
	we get from 
	\[
	f\left(\|\Lambda u(t)\|_{L^2} \|u(t)\|^{\sigc}_{L^2} \right) < (1-\theta) E_0(Q)[M(Q)]^{\sigc}
	\]
	that
	\begin{align} \label{coer-1-proo}
	\frac{3\alpha+2b}{3\alpha-4+2b} \left(\frac{\|\Lambda u(t)\|_{L^2} \|u(t)\|^{\sigc}_{L^2}}{\|\nabla Q\|_{L^2} \|Q\|^{\sigc}_{L^2}} \right)^2 - \frac{4}{3\alpha-4+2b} \left( \frac{\|\Lambda u(t)\|_{L^2} \|u(t)\|^{\sigc}_{L^2}}{\|\nabla Q\|_{L^2} \|Q\|^{\sigc}_{L^2}}\right)^{\frac{3\alpha+2b}{2}} < 1-\theta.
	\end{align}
	Consider the function $g(y):= \frac{3\alpha+2b}{3\alpha-4+2b} y^2 -\frac{4}{3\alpha-4+2b} y^{\frac{3\alpha+2b}{2}}$ with $0<y<1$. We see that $g$ is strictly increasing on $(0,1)$ and $g(0)=0, g(1)=1$. It follows from \eqref{coer-1-proo} that there exists $\rho=\rho(\theta)>0$ such that $y<1-2\rho$. The proof is complete.
\end{proof}

\begin{remark} \label{rem-coer-1}
	By the assumption $V\geq 0$ and the same argument as above, we see that if $u_0 \in H^1$ satisfies \eqref{cond-ener} and \eqref{cond-grad-glob-refi}, then 
	\[
	\|\nabla u(t)\|_{L^2} \|u(t)\|^{\sigc}_{L^2} < \|\nabla Q\|_{L^2} \|Q\|^{\sigc}_{L^2}
	\]
	for all $t$ in the existence time. In particular, the solution exists globally in time, and there exists $\rho=\rho(u_0,Q)>0$ such that
	\begin{align} \label{coer-1-refi}
	\|\nabla u(t)\|_{L^2} \|u(t)\|^{\sigc}_{L^2} < (1-2\rho) \|\nabla Q\|_{L^2} \|Q\|^{\sigc}_{L^2}
	\end{align}
	for all $t\in \R$.
\end{remark}

\begin{lemma} \label{lem-coer-2}
	Let $0<b<1$ and $\frac{4-2b}{3}<\alpha<4-2b$. Let $V: \R^3 \rightarrow \R$ satisfy \eqref{ass-V-1} and $V\geq 0$. Let $u_0 \in H^1$ satisfy \eqref{cond-ener} and \eqref{cond-grad-glob}. Let $\rho$ be as in Lemma $\ref{lem-coer-1}$. Then there exists $R_0 =R_0(\rho, u_0)>0$ such that for any $R\geq R_0$,
	\begin{align} \label{est-solu-2}
	\|\nabla (\chi_R u(t))\|_{L^2} \|\chi_R u(t)\|^{\sigc}_{L^2} < (1-\rho) \|\nabla Q\|_{L^2} \|Q\|^{\sigc}_{L^2}
	\end{align}
	for all $t \in \R$, where $\chi_R(x) = \chi(x/R)$ with $\chi \in C^\infty_0(\R^3)$ satisfying $0\leq \chi \leq 1$, $\chi(x)=1$ for $|x| \leq 1/2$ and $\chi(x)=0$ for $|x|\geq 1$. Moreover, there exists $\delta =\delta(\rho)>0$ such that for any $R\geq R_0$,
	\begin{align} \label{coer-2}
	\|\nabla (\chi_R u(t))\|^2_{L^2} - \frac{3\alpha+2b}{2(\alpha+2)} \int |x|^{-b} |\chi_R u(t,x)|^{\alpha+2} dx \geq \delta \int |x|^{-b} |\chi_R u(t,x)|^{\alpha+2} dx
	\end{align}
	for all $t\in \R$.
\end{lemma} 

\begin{proof}
	By the definition of $\chi_R$, we have $\|\chi_R u(t)\|_{L^2} \leq \|u(t)\|_{L^2}$. On the other hand, by interation by parts, we get
	\begin{align*}
	\int |\nabla(\chi f)|^2 dx &= \int \chi^2 |\nabla f|^2 dx + \int |\nabla \chi|^2 |f|^2 dx + 2\rea \int \chi \overline{f} \nabla \chi \cdot \nabla f dx \\
	&= \int \chi^2 |\nabla f|^2 dx - \int \chi \Delta \chi |f|^2 dx
	\end{align*}
	which implies
	\begin{align*}
	\|\nabla (\chi_R u(t))\|^2_{L^2} &= \int \chi^2_R |\nabla u(t)|^2 dx - \int \chi_R \Delta \chi_R |u(t)|^2 dx \\
	&\leq \|\nabla u(t)\|^2_{L^2} + O\left(R^{-2} \|u(t)\|^2_{L^2} \right).
	\end{align*}
	It follows from \eqref{coer-1}, $V\geq 0$ and the conservation of mass that
	\begin{align*}
	\|\nabla (\chi_R u(t))\|_{L^2} \|\chi_R u(t)\|^{\sigc}_{L^2} &\leq \left( \|\nabla u(t)\|^2_{L^2} + O\left(R^{-2} \|u(t)\|^2_{L^2} \right) \right)^{\frac{1}{2}} \|u(t)\|^{\sigc}_{L^2} \\
	&\leq \|\nabla u(t)\|_{L^2} \|u(t)\|^{\sigc}_{L^2} + O \left( R^{-1} \|u(t)\|^{1+{\sigc}}_{L^2}\right) \\
	&\leq \|\Lambda u(t)\|_{L^2} \|u(t)\|^{\sigc}_{L^2} + O \left( R^{-1} \|u(t)\|^{1+{\sigc}}_{L^2}\right) \\
	&\leq (1-2\rho) \|\nabla Q\|_{L^2} \|Q\|_{L^2}^{\sigc} + O \left( R^{-1} \|u_0\|_{L^2}^{1+{\sigc}}\right) \\
	&\leq (1-\rho) \|\nabla Q\|_{L^2} \|Q\|_{L^2}^{\sigc}
	\end{align*}
	provided $R \geq R_0$ with $R_0 =R_0(\rho, u_0)>0$ sufficiently large. This proves \eqref{est-solu-2}.
	
	The estimate \eqref{coer-2} follows from \eqref{est-solu-2} and the following fact: if $\|\nabla f\|_{L^2} \|f\|^{\sigc}_{L^2} <(1-\rho) \|\nabla Q\|_{L^2} \|Q\|_{L^2}^{\sigc}$, then there exists $\delta=\delta(\rho)>0$ such that
	\[
	H(f):=\|\nabla f\|^2_{L^2} -\frac{3\alpha+2b}{2(\alpha+2)} \int |x|^{-b} |f(x)|^{\alpha+2} dx \geq \delta \int |x|^{-b} |f(x)|^{\alpha+2} dx.
	\]
	To see this, we have from the Gagliardo-Nirenberg inequality, $V\geq 0$ and \eqref{GN-const} that
	\begin{align*}
	E_0(f) &=\frac{1}{2} \|\nabla f\|^2_{L^2} - \frac{1}{\alpha+2} \int |x|^{-b} |f(x)|^{\alpha+2} dx \\
	&\geq \frac{1}{2} \|\nabla f\|^2_{L^2} - \frac{C_{\opt}}{\alpha+2} \|\nabla f\|^{\frac{3\alpha+2b}{2}}_{L^2} \|f\|^{\frac{4-2b-\alpha}{2}}_{L^2} \\
	&= \frac{1}{2} \|\nabla f\|^2_{L^2} \left( 1- \frac{2 C_{\opt}}{\alpha+2} \|\nabla f\|^{\frac{3\alpha-4+2b}{2}}_{L^2} \|f\|^{\frac{4-2b-\alpha}{2}}_{L^2} \right) \\
	&\geq \frac{1}{2} \|\nabla f\|^2_{L^2} \left( 1- \frac{2 C_{\opt}}{\alpha+2} \|\Lambda f\|^{\frac{3\alpha-4+2b}{2}}_{L^2} \|f\|^{\frac{4-2b-\alpha}{2}}_{L^2} \right) \\
	&= \frac{1}{2} \|\nabla f\|^2_{L^2} \left( 1- \frac{2 C_{\opt}}{\alpha+2} \left(\|\Lambda f\|_{L^2} \|f\|^{\sigc}_{L^2} \right)^{\frac{3\alpha-4+2b}{2}} \right) \\
	&> \frac{1}{2} \|\nabla f\|^2_{L^2} \left( 1- \frac{2C_{\opt}}{\alpha+2} (1-\rho)^{\frac{3\alpha-4+2b}{2}} \left( \|\nabla Q\|_{L^2} \|Q\|^{\sigc}_{L^2} \right)^{\frac{3\alpha-4+2b}{2}} \right) \\
	&= \frac{1}{2} \|\nabla f\|^2_{L^2} \left( 1-\frac{4}{3\alpha+2b} (1-\rho)^{\frac{3\alpha-4+2b}{2}} \right).
	\end{align*}
	It follows that
	\[
	\|\nabla f\|^2_{L^2} >\frac{3\alpha+2b}{2(\alpha+2)} \frac{1}{(1-\rho)^{\frac{3\alpha-4+2b}{2}}} \int |x|^{-b} |f(x)|^{\alpha+2} dx.
	\]
	We thus get
	\begin{align*}
	H(f) &= \frac{3\alpha+2b}{2} E_0(f) - \frac{3\alpha-4+2b}{4} \|\nabla f\|^2_{L^2} \\
	&> \frac{3\alpha+2b}{4} \|\nabla f\|^2_{L^2} \left( 1 - \frac{4}{3\alpha+2b} (1-\rho)^{\frac{3\alpha-4+2b}{2}} \right) - \frac{3\alpha -4+2b}{4} \|\nabla f\|^2_{L^2} \\
	&= \left(1- (1-\rho)^{\frac{3\alpha-4+2b}{2}} \right) \|\nabla f\|^2_{L^2} \\
	&> \frac{(3\alpha +2b)\left[1-(1-\rho)^{\frac{3\alpha-4+2b}{2}} \right] }{2(\alpha+2)(1-\rho)^{\frac{3\alpha-4+2b}{2}}} \int |x|^{-b} |f(x)|^{\alpha+2} dx
	\end{align*}
	which proves the fact. The proof is complete.
\end{proof}

	\begin{remark} \label{rem-coer-2}
	By Remark $\ref{rem-coer-1}$ and the same argument as above, we see that the estimates \eqref{est-solu-2} and \eqref{coer-2} still hold if we assume $u_0 \in H^1$ satisfying \eqref{cond-ener} and \eqref{cond-grad-glob-refi}.
	\end{remark}
	
\subsection{Morawetz estimate}
Let us start with the following virial identity.
\begin{lemma}[Virial identity \cite{Farah, Dinh-blow}] \label{lem-virial-iden}
	Let $0<b<1$ and $0<\alpha<4-2b$. Let $V: \R^3 \rightarrow \R$ satisfy \eqref{ass-V-1} and \eqref{ass-V-2}. Let $\varphi: \R^3 \rightarrow \R$ be a sufficiently smooth and decaying function. Let $u$ be a solution to \eqref{INLS-V}. Define
	\[
	M_\varphi(t):= 2 \int \nabla \varphi \cdot \ima \left( \overline{u}(t) \nabla u(t)\right) dx.
	\]
	Then it holds that
	\begin{align*}
	\frac{d}{dt} M_\varphi(t) = -\int \Delta^2 \varphi |u(t)|^2 dx &+ 4 \sum_{j,k=1}^3 \int \partial^2_{jk} \varphi \rea \left( \partial_j \overline{u}(t) \partial_k u(t)\right) dx -2 \int \nabla \varphi \cdot \nabla V |u(t)|^2 dx\\
	&\pm \frac{2\alpha}{\alpha+2} \int |x|^{-b} \Delta \varphi |u(t)|^{\alpha+2} dx \pm \frac{4b}{\alpha+2} \int |x|^{-b-2} x \cdot \nabla \varphi |u(t)|^{\alpha+2} dx.
	\end{align*}
\end{lemma}

\begin{remark} \label{rem-virial-iden}
	\begin{itemize}
		\item In the case $\varphi(x) = |x|^2$, we have
		\begin{align} \label{viri-iden}
		\frac{d}{dt} \|x u(t)\|^2_{L^2} = M_{|x|^2} (t), \quad \frac{d^2}{dt^2} \|x u(t)\|^2_{L^2} = 8 K(u(t)),
		\end{align}
		where
		\begin{align} \label{defi-K}
		K(u(t)) := \|\nabla u(t)\|^2_{L^2} - \frac{1}{2} \int x \cdot  \nabla V |u(t)|^2 dx - \frac{3\alpha+2b}{2(\alpha+2)} \int |x|^{-b} |u(t)|^{\alpha+2} dx.
		\end{align}
		\item In the case $\varphi$ is radially symmetric, by using the fact
		\[
		\partial_j = \frac{x_j}{r} \partial_r, \quad \partial^2_{jk} = \left( \frac{\delta_{jk}}{r} - \frac{x_jx_k}{r^3} \right) \partial_r + \frac{x_j x_k}{r^2} \partial^2_r,
		\]
		we see that
		\[
		\sum_{j,k=1}^3 \int \partial^2_{jk} \varphi \rea (\partial_j \overline{u}(t) \partial_k u(t)) dx = \int \frac{\varphi'(r)}{r} |\nabla u(t)|^2 dx + \int \left( \frac{\varphi''(r)}{r^2} - \frac{\varphi'(r)}{r^3}\right) |x \cdot \nabla u(t)|^2 dx.
		\]
		In particular, 
		\begin{align*}
		\frac{d}{dt} M_\varphi(t) = &- \int \Delta^2 \varphi |u(t)|^2 dx + 4 \int \frac{\varphi'(r)}{r} |\nabla u(t)|^2 dx \\
		&+ 4\int \left( \frac{\varphi''(r)}{r^2} - \frac{\varphi'(r)}{r^3}\right) |x \cdot \nabla u(t)|^2 dx - 2\int \frac{\varphi'(r)}{r} x \cdot \nabla V |u(t)|^2 dx \\
		& \pm \frac{2\alpha}{\alpha+2} \int |x|^{-b} \Delta \varphi |u(t)|^{\alpha+2} dx \pm \frac{4b}{\alpha+2} \int |x|^{-b} \frac{\varphi'(r)}{r} |u(t)|^{\alpha+2} dx.
		\end{align*}
	\end{itemize}
\end{remark}

Let $\zeta: [0,\infty) \rightarrow [0,2]$ be a smooth function satisfying
\begin{align*} 
\zeta(r) = \left\{
\begin{array}{ccl}
2 & \text{if} & 0\leq r \leq 1, \\
0 &\text{if} & r\geq 2.
\end{array}
\right.
\end{align*}
We define the function $\theta: [0,\infty) \rightarrow [0,\infty)$ by
\[
\theta(r):= \int_0^r \int_0^s \zeta(z)dz ds.
\]
Given $R>0$, we define a radial function
\begin{align} \label{def-varphi-R}
\varphi_R(x) = \varphi_R(r):= R^2 \theta(r/R), \quad r=|x|.
\end{align}
It is easy to check that
\[
2 \geq \varphi''_R(r) \geq 0, \quad 2-\frac{\varphi'(r)}{r} \geq 0, \quad 6-\Delta \varphi_R(x) \geq 0, \quad \forall r \geq 0, \quad \forall x \in \R^3.
\]
%We also have that
%\[
%\|\nabla^k \varphi_R\|_{L^\infty} \lesssim R^{2-k}, \quad k=0,\cdots, 4
%\]
%and
%\[
%\supp(\nabla^k \varphi_R) \subset \left\{
%\begin{array}{cl}
%\{|x| \leq 2R\} &\text{if } k=1,2, \\
%\{R \leq |x| \leq 2R\} &\text{if } k=3,4.
%\end{array}
%\right.
%\]

\begin{proposition} \label{prop-mora-est}
	Let $0<b<1$ and $\frac{4-2b}{3}<\alpha<4-2b$. Let $V: \R^3 \rightarrow \R$ satisfy \eqref{ass-V-1}, $V\geq 0$, $x \cdot \nabla V \leq 0$, $x \cdot \nabla V \in L^{\frac{3}{2}}$ and $V$ be radially symmetric. Let $u_0 \in H^1$ be radially symmetric and satisfy \eqref{cond-ener} and \eqref{cond-grad-glob}. Then for any $T>0$ and any $R\geq R_0$ with $R_0$ as in Lemma $\ref{lem-coer-2}$, the corresponding global solution to the focusing problem \eqref{INLS-V} satisfies
	\begin{align} \label{mora-est}
	\frac{1}{T}\int_0^T \int_{|x| \leq R/2} |x|^{-b} |u(t,x)|^{\alpha+2} dx dt \leq C(u_0,Q) \left[\frac{R}{T} + \frac{1}{R^2} +\frac{1}{R^{\alpha+b}} + o_R(1) \right]
	\end{align}
	for some constant $C(u_0,Q)$ depending only on $u_0$ and $Q$, where $A_R=o_R(1)$ means $A_R \rightarrow 0$ as $R \rightarrow \infty$. In particular,
	\begin{align} \label{mora-est-focus}
	\frac{1}{T} \int_0^T \int_{|x| \leq R/2} |u(t,x)|^{\alpha+2+b} dx dt \leq C(u_0,Q) \left[\frac{R}{T} +\frac{1}{R^2} +\frac{1}{R^{\alpha+b}} + o_R(1) \right].
	\end{align}
\end{proposition}

\begin{proof}
	Let $\delta=\delta(u_0,Q)$ be as in \eqref{coer-2}, and $R_0=R_0(\rho,u_0)$ be as in Lemma $\ref{lem-coer-2}$. Let $\varphi_R$ be as in \eqref{def-varphi-R}. By the Cauchy-Schwarz inequality, the conservation of mass and \eqref{est-solu-1}, we see that
	\begin{align} \label{mora-est-proo}
	|M_{\varphi_R}(t)| \leq \|\nabla \varphi_R\|_{L^\infty} \|u(t)\|_{L^2} \|\nabla u(t)\|_{L^2} \leq \|\nabla \varphi_R\|_{L^\infty} \|u(t)\|_{L^2} \|\Lambda u(t)\|_{L^2} \lesssim R
	\end{align}
	for all $t\in \R$, where the implicit constant depends only on $u_0$ and $Q$. By Lemma $\ref{lem-virial-iden}$ and the fact $\varphi_R(x)=|x|^2$ for $|x| \leq R$, 
	\begin{align*}
	\frac{d}{dt} M_{\varphi_R}(t) &= - \int \Delta^2 \varphi_R |u(t)|^2 dx + 4 \sum_{j,k=1}^3 \int \partial^2_{jk} \varphi_R \rea \left( \partial_j \overline{u}(t) \partial_k u(t) \right) dx - 2 \int \nabla \varphi_R \cdot \nabla V |u(t)|^2 dx\\
	& \mathrel{\phantom{=}} - \frac{2\alpha}{\alpha+2} \int |x|^{-b} \Delta \varphi_R |u(t)|^{\alpha+2} dx - \frac{4b}{\alpha+2} \int |x|^{-b-2} x \cdot \nabla \varphi_R |u(t)|^{\alpha+2} dx \\
	&= 8 \left( \int_{|x| \leq R} |\nabla u(t)|^2 dx - \frac{N\alpha+2b}{2(\alpha+2)} \int_{|x| \leq R} |x|^{-b} |u(t)|^{\alpha+2} dx -\frac{1}{2} \int_{|x| \leq R} x \cdot \nabla V |u(t)|^2 dx\right) \\
	&\mathrel{\phantom{=}} - \int \Delta^2 \varphi_R |u(t)|^2 dx + 4 \sum_{j,k=1}^N \int_{|x|>R} \partial^2_{jk} \varphi_R \rea \left( \partial_j \overline{u}(t) \partial_k u(t) \right) dx \\
	&\mathrel{\phantom{=}} -2 \int_{|x|>R} \nabla \varphi_R \cdot \nabla V |u(t)|^2 dx - \frac{2\alpha}{\alpha+2} \int_{|x|>R} |x|^{-b} \Delta \varphi_R |u(t)|^{\alpha+2} dx \\
	&\mathrel{\phantom{=-2 \int_{|x|>R} \nabla \varphi_R \cdot \nabla V |u(t)|^2 dx}} - \frac{4b}{\alpha+2} \int_{|x|>R} |x|^{-b-2} x \cdot \nabla \varphi_R |u(t)|^{\alpha+2} dx.
	\end{align*}
	Since $\|\Delta^2 \varphi_R\|_{L^\infty} \lesssim R^{-2}$, the conservation of mass implies
	\[
	\int \Delta^2 \varphi_R |u(t)|^2 dx \lesssim R^{-2}.
	\]
	Since $u$ is radial, we use the fact
	\[
	\partial^2_{jk} = \left(\frac{\delta_{jk}}{r} - \frac{x_j x_k}{r^3} \right) \partial_r + \frac{x_j x_k}{r^2} \partial^2_r
	\]
	to get
	\[
	\sum_{j,k=1}^N \partial^2_{jk} \varphi_R \partial_j \overline{u} \partial_k u = \varphi''_R |\partial_r u|^2 \geq 0
	\]
	which implies
	\[
	\int_{|x|>R} \partial^2_{jk} \varphi_R \rea \left( \partial_j \overline{u}(t) \partial_k u(t) \right) dx \geq 0.
	\]
	Using the fact $|\nabla \varphi_R \cdot \nabla V| = \varphi'_R \partial_r V \leq 2|x \cdot \nabla V|$ and $x \cdot \nabla V \in L^{\frac{3}{2}}$, the Sobolev embedding implies
	\begin{align*}
	\left| \int_{|x|>R}  \nabla \varphi_R \cdot \nabla V |u(t)|^2 dx \right| &\lesssim \int_{|x|>R} |x \cdot \nabla V| |u(t)|^2 dx \\
	&\leq \|x \cdot \nabla V\|_{L^{\frac{3}{2}}(|x|>R)} \|u(t)\|^2_{L^6} \\
	&\lesssim \|x \cdot \nabla V\|_{L^{\frac{3}{2}}(|x|>R)} \|\Lambda u(t)\|^2_{L^2} = o_R(1).
	\end{align*}
	Since $\|\Delta \varphi_R\|_{L^\infty} \lesssim 1$ and $\|x \cdot \nabla \varphi_R\|_{L^\infty} \lesssim |x|^2$, the radial Sobolev embedding (see e.g. \cite{Strauss}):
	\begin{align} \label{rad-sob-emb}
	\||x| f\|_{L^\infty} \lesssim \|f\|_{H^1}, \quad \forall f \in H^1_{\rad}(\R^3)
	\end{align}
	implies that
	\begin{align*}
	\left|\int_{|x|>R} \left(|x|^{-b} \Delta \varphi_R + |x|^{-b-2} x\cdot \nabla \varphi_R\right) |u(t)|^{\alpha+2} dx \right| &\lesssim \left(\sup_{|x|>R} |x|^{-b} |u(t,x)|^\alpha \right) \|u(t)\|^2_{L^2} \\
	&\lesssim R^{- \alpha -b} \|u(t)\|^\alpha_{H^1} \|u(t)\|^2_{L^2} \lesssim R^{-\alpha -b}.
	\end{align*}
	It follows that
	\begin{align*}
	\frac{d}{dt} M_{\varphi_R}(t) &\geq 8 \Big(\int_{|x| \leq R} |\nabla u(t)|^2 dx - \frac{3\alpha+2b}{2(\alpha+2)} \int_{|x| \leq R} |x|^{-b} |u(t)|^{\alpha+2} dx - \frac{1}{2} \int_{|x| \leq R} x \cdot \nabla V |u(t)|^2 dx \Big) \\
	&\mathrel{\phantom{\geq 8 \Big(\int_{|x| \leq R} |\nabla u(t)|^2 dx - \frac{3\alpha+2b}{2(\alpha+2)} \int_{|x| \leq R} |x|^{-b} |u(t)|^{\alpha+2} dx}}+ O \left( R^{-2} + R^{-\alpha -b} + o_R(1)\right) \\
	& \geq 8 \Big(\int_{|x| \leq R} |\nabla u(t)|^2 dx - \frac{3\alpha+2b}{2(\alpha+2)} \int_{|x| \leq R} |x|^{-b} |u(t)|^{\alpha+2} dx \Big) + O \left( R^{-2} + R^{-\alpha -b} + o_R(1)\right),		
	\end{align*}
	where the second line follows from the fact $x\cdot \nabla V \leq 0$. On the other hand, let $\chi_R$ be as in Lemma $\ref{lem-coer-2}$. We see that
	\begin{align*}
	\int |\nabla(\chi_R u(t))|^2 dx &= \int \chi^2_R |\nabla u(t)|^2 dx - \int \chi_R \Delta(\chi_R) |u(t)|^2 dx \\
	&= \int_{|x| \leq R} |\nabla u(t)|^2 dx - \int_{R/2 \leq |x| \leq R} (1-\chi^2_R) |\nabla u(t)|^2 dx - \int \chi_R \Delta(\chi_R) |u(t)|^2 dx
	\end{align*}
	and
	\[
	\int |x|^{-b}|\chi_R u(t)|^{\alpha+2} dx = \int_{|x| \leq R} |x|^{-b} |u(t)|^{\alpha+2} dx - \int_{R/2 \leq |x| \leq R} (1-\chi_R^{\alpha+2}) |x|^{-b} |u(t)|^{\alpha+2} dx.
	\]
	It follows that
	\begin{align*}
	\int_{|x| \leq R} |\nabla u(t)|^2 dx &- \frac{3\alpha+2b}{2(\alpha+2)} \int_{|x| \leq R} |x|^{-b} |u(t)|^{\alpha+2} dx \\
	&= \int |\nabla(\chi_R u(t))|^2 dx - \frac{3\alpha+2b}{2(\alpha+2)} \int |x|^{-b} |\chi_R u(t)|^{\alpha+2} dx \\
	&\mathrel{\phantom{=}} + \int_{R/2 \leq |x| \leq R} (1-\chi^2_R) |\nabla u(t)|^2 dx + \int \chi_R \Delta (\chi_R) |u(t)|^2 dx \\
	&\mathrel{\phantom{=}} - \frac{3\alpha+2b}{2(\alpha+2)} \int_{R/2 \leq |x| \leq R} (1-\chi_R^{\alpha+2}) |x|^{-b} |u(t)|^{\alpha+2} dx.
	\end{align*}
	Thanks to the fact $0\leq \chi_R \leq 1$, $\|\Delta(\chi_R)\|_{L^\infty} \lesssim R^{-2}$ and the radial Sobolev embedding, we infer that
	\begin{align*}
	\int_{|x| \leq R} |\nabla u(t)|^2 dx &- \frac{3\alpha+2b}{2(\alpha+2)} \int_{|x| \leq R} |x|^{-b} |u(t)|^{\alpha+2} dx \\
	&\geq \int |\nabla(\chi_R u(t))|^2 dx - \frac{3\alpha+2b}{2(\alpha+2)} \int |x|^{-b} |\chi_R u(t)|^{\alpha+2} dx + O\left(R^{-2} + R^{-\alpha -b} \right).
	\end{align*}
	We thus obtain
	\[
	\frac{d}{dt} M_{\varphi_R}(t) \geq 8 \left( \int |\nabla(\chi_R u(t))|^2 dx - \frac{3\alpha+2b}{2(\alpha+2)} \int |x|^{-b} |\chi_R u(t)|^{\alpha+2} dx \right) + O\left(R^{-2} + R^{-\alpha -b} +o_R(1)\right).
	\]
	By Lemma $\ref{lem-coer-2}$, there exist $\delta=\delta(\rho)>0$ and $R_0=R_0(\rho,u_0)>0$ such that for any $R\geq R_0$,
	\[
	8 \delta \int |x|^{-b} |\chi_R u(t)|^{\alpha+2} dx \leq \frac{d}{dt} M_{\varphi_R}(t) + O\left(R^{-2} + R^{-\alpha -b} +o_R(1)\right)
	\]
	which, by \eqref{mora-est-proo}, implies 
	\[
	8 \delta \int_0^T \int |x|^{-b} |\chi_R u(t)|^{\alpha+2} dx dt \leq R + O\left(R^{-2} + R^{-\alpha -b} +o_R(1)\right) T.
	\]
	By the definition of $\chi_R$, we obtain
	\begin{align*} 
	\frac{1}{T} \int_0^T \int_{|x| \leq R/2} |x|^{-b} |u(t,x)|^{\alpha+2} dx dt \lesssim \frac{R}{T} + \frac{1}{R^2} + \frac{1}{R^{\alpha +b}} + o_R(1)
	\end{align*}
	which proves \eqref{mora-est}. To see \eqref{mora-est-focus}, we use the radial Sobolev embedding to get
	\begin{align*}
	\frac{1}{T} \int_0^T \int_{|x| \leq R/2} |u(t)|^{\alpha+2 +b} dx dt  &= \frac{1}{T} \int_0^T \int_{|x| \leq R/2} (|x| |u(t)|)^b |x|^{-b} |u(t)|^{\alpha+2} dx dt \\
	& \lesssim \|u\|^b_{L^\infty([0,T], H^1)} \frac{1}{T} \int_0^T \int_{|x| \leq R/2} |x|^{-b} |u(t)|^{\alpha+2} dx dt \\
	& \lesssim \frac{R}{T} + \frac{1}{R^2} + \frac{1}{R^{\alpha +b}} + o_R(1).
	\end{align*}
	The proof is complete.
\end{proof}

\begin{remark} \label{rem-mora-est-refi}
	Using Remark $\ref{rem-coer-2}$, we see that Proposition $\ref{prop-mora-est}$ still holds if we assume \eqref{cond-grad-glob-refi} in place of \eqref{cond-grad-glob}. 
\end{remark}

\begin{corollary} \label{coro-mora-est-focus}
	Let $0<b<1$ and $\frac{4-2b}{3}<\alpha<4-2b$. Let $V: \R^3 \rightarrow \R$ satisfy \eqref{ass-V-1}, $V\geq 0$, $x \cdot \nabla V \leq 0$, $x \cdot \nabla V \in L^{\frac{3}{2}}$ and $V$ be radially symmetric. Let $u_0 \in H^1$ be radially symmetric and satisfy \eqref{cond-ener} and \eqref{cond-grad-glob}. Then there exists $t_n \rightarrow + \infty$ such that the corresponding global solution to the focusing problem \eqref{INLS-V} satisfies for any $R>0$,
	\begin{align} \label{small-L2}
	\lim_{n\rightarrow \infty} \int_{|x| \leq R} |u(t_n,x)|^2 dx =0.
	\end{align}
\end{corollary}

\begin{proof}
	Applying \eqref{mora-est-focus} with $T= R^3$, we have for $R \geq R_0$,
	\[
	\frac{1}{R^3} \int_0^{R^3} \int_{|x| \leq R/2} |u(t,x)|^{\alpha+2+b} dx dt \leq C(u_0,Q) \left[ \frac{1}{R^2} +\frac{1}{R^{\alpha+b}} + o_R(1)\right].
	\]
	The fundamental theorem of calculus then implies that there exist sequences $t_n \rightarrow +\infty$ and $R_n \rightarrow \infty$ such that
	\begin{align} \label{small-L2-proo}
	\lim_{n\rightarrow \infty} \int_{|x| \leq R_n/2} |u(t_n,x)|^{\alpha+2+b} dx =0.
	\end{align}
	Now let $R>0$. By taking $n$ sufficiently large, we have $R_n/2 \geq R$. By H\"older's inequality and \eqref{small-L2-proo},
	\begin{align*}
	\int_{|x| \leq R} |u(t_n,x)|^2 dx &\leq \left( \int_{|x| \leq R} dx \right)^{\frac{\alpha+b}{\alpha+2+b}} \left( \int_{|x| \leq R} |u(t_n,x)|^{\alpha+2+b} dx \right)^{\frac{2}{\alpha+2+b}} \\
	&\lesssim R^{\frac{3(\alpha+b)}{\alpha+2+b}} \left( \int_{|x| \leq R_n/2} |u(t_n,x)|^{\alpha+2+b} dx \right)^{\frac{2}{\alpha+2+b}} \\
	&\rightarrow 0
	\end{align*}
	as $n\rightarrow \infty$. 
\end{proof}

\begin{remark} \label{rem-coro-mora-est-refi}
	As in Remark $\ref{rem-mora-est-refi}$, Corollary $\ref{coro-mora-est-focus}$ still holds if we assume \eqref{cond-grad-glob-refi} instead of \eqref{cond-grad-glob}. 
\end{remark}

\begin{corollary} \label{coro-mora-est-defocus}
	Let $0<b<1$ and $\frac{4-2b}{3}<\alpha <4-2b$. Let $V: \R^3 \rightarrow \R$ satisfy \eqref{ass-V-1}, \eqref{ass-V-2}, $x\cdot \nabla V \in L^{\frac{3}{2}}$, $x \cdot \nabla V \leq 0$ and $V$ be radially symmetric. Let $u_0 \in H^1$ be radially symmetric. Then there exists $t_n \rightarrow +\infty$ such that the corresponding global solution to the defocusing problem \eqref{INLS-V} satisfies for any $R>0$,
	\begin{align} \label{small-L2-defocus}
	\lim_{n\rightarrow \infty} \int_{|x| \leq R} |u(t_n,x)|^2 dx =0.
	\end{align}
\end{corollary}

\begin{proof}
	The proof is similar to the ones of Proposition $\ref{prop-mora-est}$ and Corollary $\ref{coro-mora-est-focus}$. We only point out the differences. We first note that if $u_0 \in H^1$, then the corresponding solution to the defocusing problem \eqref{INLS-V} exists globally in time and satisfies $\|u(t)\|_{H^1} \leq C(E,M)$ for all $t\in \R$. We also have
	\begin{align*}
	\frac{d}{dt} M_{\varphi_R}(t) &= - \int \Delta^2 \varphi_R |u(t)|^2 dx + 4 \sum_{j,k=1}^3 \int \partial^2_{jk} \varphi_R \rea (\partial_j \overline{u}(t) \partial_k u(t)) dx - 2 \int \nabla \varphi_R \cdot \nabla V |u(t)|^2 dx \\
	&\mathrel{\phantom{=}} +\frac{2\alpha}{\alpha+2} \int |x|^{-b} \Delta \varphi_R |u(t)|^{\alpha+2} dx +\frac{4b}{\alpha+2} \int |x|^{-b-2} x\cdot \nabla \varphi_R |u(t)|^{\alpha+2} dx \\
	&= 8 \left( \int_{|x| \leq R} |\nabla u(t)|^2 dx + \frac{3\alpha+2b}{2(\alpha+2)} \int_{|x| \leq R} |x|^{-b} |u(t)|^{\alpha+2} dx - \frac{1}{2} \int_{|x| \leq R} x \cdot \nabla V |u(t)|^2 dx \right) \\
	&\mathrel{\phantom{=}} - \int \Delta^2 \varphi_R |u(t)|^2 dx + 4 \sum_{j,k=1}^3 \int_{|x|>R} \partial^2_{jk} \varphi_R \rea (\partial_j \overline{u}(t) \partial_k u(t)) dx \\
	&\mathrel{\phantom{=}} - 2 \int_{|x|>R} \nabla \varphi_R \cdot \nabla V |u(t)|^2 dx + \frac{2\alpha}{\alpha+2} \int_{|x|>R} |x|^{-b} \Delta \varphi_R |u(t)|^{\alpha+2} dx \\
	&\mathrel{\phantom{=- 2 \int_{|x|>R} \nabla \varphi_R \cdot \nabla V |u(t)|^2 dx }} + \frac{4b}{\alpha+2} \int_{R \leq |x| \leq 2R} |x|^{-b-2} x \cdot \nabla \varphi_R |u(t)|^{\alpha+2} dx.
	\end{align*}
	Estimating as above, we get
	\[
	\frac{d}{dt} M_{\varphi_R}(t) \geq \frac{4(3\alpha+2b)}{\alpha+2} \int_{|x| \leq R} |x|^{-b} |u(t)|^{\alpha+2} dx + O\left(R^{-2} + R^{-\alpha-b} + o_R(1)\right).
	\]
	Using the fact $0\leq \chi_R \leq 1$ and
	\[
	\int |x|^{-b} |\chi_R u(t)|^{\alpha+2} dx = \int_{|x| \leq R} |x|^{-b} |u(t)|^{\alpha+2} dx - \int_{R/2 \leq |x| \leq R} (1-\chi_R^{\alpha+2}) |x|^{-b} |u(t)|^{\alpha+2} dx,
	\]
	the radial Sobolev embedding implies
	\[
	\frac{d}{dt} M_{\varphi_R}(t) \geq \frac{3\alpha+2b}{2(\alpha+2)} \int |x|^{-b} |\chi_R u(t)|^{\alpha+2} dx + O\left( R^{-2} + R^{-\alpha-b} +o_R(1)\right).
	\]
	Repeating the same argument as in the proof of Proposition $\ref{prop-mora-est}$ and Corollary $\ref{coro-mora-est-focus}$, we complete the proof.
\end{proof}

\subsection{Nonlinear estimates}
Let us start with the following nonlinear estimates.
\begin{lemma} \label{lem-non-est}
	Let $0<b<1$, $\frac{4-2b}{3}<\alpha<3-2b$ and $I \subset \R$. Then there exist $\theta, \theta_1, \theta_2 \in (0,1)$ such that
	\begin{align*}
	\||x|^{-b} |u|^\alpha u\|_{S'(\dot{H}^{-\gamc},I)} &\lesssim \|u\|^{\alpha \theta}_{L^\infty(I,H^1)} \|u\|^{1+\alpha(1-\theta)}_{S(\dot{H}^{\gamc},I)}, \\
	\||x|^{-b} |u|^\alpha u\|_{L^2(I, L^{\frac{6}{5}})} &\lesssim \left(\|u\|^{\alpha \theta_1}_{L^\infty(I, H^1)} \|u\|^{\alpha(1-\theta_1)}_{S(\dot{H}^{\gamc},I)} + \|u\|^{\alpha \theta_2}_{L^\infty(I, H^1)} \|u\|^{\alpha(1-\theta_2)}_{S(\dot{H}^{\gamc},I)} \right) \|u\|_{S(L^2,I)}, \\
	\|\nabla (|x|^{-b} |u|^\alpha u)\|_{L^2(I, L^{\frac{6}{5}})} &\lesssim \left(\|u\|^{\alpha \theta_1}_{L^\infty(I, H^1)} \|u\|^{\alpha(1-\theta_1)}_{S(\dot{H}^{\gamc},I)} +\|u\|^{\alpha \theta_2}_{L^\infty(I, H^1)} \|u\|^{\alpha(1-\theta_2)}_{S(\dot{H}^{\gamc},I)} \right) \|\nabla u\|_{S(L^2,I)},
	\end{align*}
	where
	\[
	\|u\|_{S(\dot{H}^{\gamc}, I)}:= \sup_{(k,l) \in S_{\gamc}} \|u\|_{L^k(I,L^l)}, \quad \|u\|_{S'(\dot{H}^{-\gamc},I)}:= \inf_{(m,n) \in S_{-\gamc}} \|u\|_{L^{m'}(I,L^{n'})}.
	\]
	Here $S_{\gamc}$ is the set of $\dot{H}^{\gamc}$-admissible pairs, i.e. 
	\begin{align} \label{def-gamc-adm}
	\frac{2}{k} + \frac{3}{l} = \frac{2-b}{\alpha}, \quad \frac{3\alpha}{2-b}^+ \leq l  \leq 6^-
	\end{align}
	and $S_{-\gamc}$ is the set of $\dot{H}^{-\gamc}$-admissible pairs, i.e.
	\[
	\frac{2}{m} +\frac{3}{n} = 3-\frac{2-b}{\alpha}, \quad \frac{3\alpha}{2-b}^+ \leq n \leq 6^-,
	\]
	where $a^+:=a+\vartheta$ for some $0<\vartheta \ll 1$ and similarly for $a^-$.
\end{lemma}

\begin{remark} \label{rem-non-est}
\begin{itemize}
	\item In the definition of $S_{\gamc}$, we do not consider the pair $\left(\infty, \frac{3\alpha}{2-b}\right)$ which is needed in \eqref{linear-small}.
	\item It is clear from the proof that the restriction $\alpha <3-2b$ is due to the equivalence norm $\|\nabla u\|_{L^r} \sim \|\Lambda u\|_{L^r}$ with $1<r<3$. 
	\item In \cite[Section 4]{Guzman} and \cite[Lemma 2.7]{Campos}, similar estimates as in Lemma $\ref{lem-non-est}$ were proved.
	%\begin{align*}
	%\||x|^{-b} |u|^\alpha u\|_{S'(\dot{H}^{-\gamc},I)} &\lesssim \|u\|^\theta_{L^\infty(I,H^1)} \|u\|^{\alpha+1-\theta}_{S(\dot{H}^{\gamc},I)}, \\
	%\||x|^{-b} |u|^\alpha u\|_{S'(L^2,I)} &\lesssim \|u\|^\theta_{L^\infty(I,H^1)} \|u\|^{\alpha-\theta}_{S(\dot{H}^{\gamc},I)} \|u\|_{S(L^2,I)}, \\
	%\|\nabla(|x|^{-b} |u|^\alpha u)\|_{S'(L^2,I)} &\lesssim \|u\|^\theta_{L^\infty(I,H^1)} \|u\|^{\alpha-\theta}_{S(\dot{H}^{\gamc},I)} \|\nabla u\|_{S(L^2,I)}.
	%\end{align*} 
	More precisely, the following estimates hold
	\begin{align}
	\||x|^{-b} |u|^\alpha u\|_{L^{m'}(I,L^{r'})} &\lesssim \|u\|^\theta_{L^\infty(I,H^1)} \|u\|^{\alpha+1-\theta}_{L^k(I,L^r)}, \label{est-1}\\
	\||x|^{-b} |u|^\alpha u\|_{L^{q'}(I,L^{r'})} &\lesssim \|u\|^\theta_{L^\infty(I,H^1)} \|u\|^{\alpha-\theta}_{L^k(I,L^r)} \|u\|_{L^q(I,L^r)}, \nonumber \\
	\|\nabla(|x|^{-b} |u|^\alpha u)\|_{L^{q'}(I,L^{r'})} &\lesssim \|u\|^\theta_{L^\infty(I,H^1)} \|u\|^{\alpha-\theta}_{L^k(I,L^r)} \|\nabla u\|_{L^q(I,L^r)}, \nonumber
	\end{align}
	where
	\begin{align*}
	q&=\frac{4\alpha(\alpha+2-\theta)}{\alpha(3\alpha+2b) -\theta(3\alpha -4+2b)}, & r&=\frac{3\alpha(\alpha+2-\theta)}{\alpha(3-b) -\theta(2-b)}, \\
	k&=\frac{2\alpha(\alpha+2-\theta)}{4-2b-\alpha}, & m &=\frac{2\alpha(\alpha+2-\theta)}{\alpha(3(\alpha-\theta) +1+2b) - (4-2b)(1-\theta)}
	\end{align*}
	with $0<\theta \ll 1$. Note that $(q,r) \in S_0, (k,r) \in S_{\gamc}$ and $(m,r) \in S_{-\gamc}$. However, it is easy to check that $r$ does not satisfy $2\leq r<3$. Thus these nonlinear estimates are not suitable for our purpose.
\end{itemize}
\end{remark}

\noindent {\it Proof of Lemma $\ref{lem-non-est}$.} 
\textcolor{blue}{The first estimate was proved in \cite[Section 4]{Guzman} and \cite[Lemma 2.7]{Campos} (see \eqref{est-1}).}
Let us prove the first estimate. We estimate
\[
\||x|^{-b} |u|^\alpha u\|_{S'(\dot{H}^{-\gamc},I)} \leq \||x|^{-b} |u|^\alpha u\|_{L^{m'}(I,L^{n'})} \leq \||x|^{-b} |u|^\alpha u\|_{L^{m'}(I,L^{n'}(B))} +\||x|^{-b} |u|^\alpha u\|_{L^{m'}(I,L^{n'}(B^c))}
\]
for some $(m,n) \in S_{-\gamc}$. By H\"older's inequality,
\begin{align}
\||x|^{-b} |u|^\alpha u\|_{L^{m'}(I,L^{n'}(B))} &\leq \||x|^{-b}\|_{L^\nu(B)} \||u|^\alpha u\|_{L^{m'}(I,L^\rho)} \nonumber \\
&\lesssim \|u\|^\alpha_{L^{\alpha q}(I,L^{\alpha r})} \|u\|_{L^k(I,L^l)} \nonumber \\
&\lesssim \|u\|^{\alpha \theta}_{L^\infty(I,L^6)} \|u\|^{\alpha (1-\theta)}_{L^d(I,L^e)} \|u\|_{L^k(I,L^l)} \label{non-est-1}
\end{align}
provided $\nu, \rho, q,r \geq 1$ and $(d,e), (k,l) \in S_{\gamc}$ satisfying
\[
\frac{1}{n'} = \frac{1}{\nu} + \frac{1}{\rho}, \quad \frac{3}{\nu}>b, \quad \frac{1}{\rho} = \frac{1}{r} + \frac{1}{l}
\]
and
\[
\frac{1}{m'} = \frac{1}{q} + \frac{1}{k}, \quad \frac{1}{\alpha q} = \frac{\theta}{\infty} + \frac{1-\theta}{d}, \quad \frac{1}{\alpha r} = \frac{\theta}{6} + \frac{1-\theta}{e}.
\]
Since $(m,n) \in S_{-\gamc}$, we infer that
\[
2+\frac{2-b}{\alpha} = \frac{2}{m'} + \frac{2}{n'} = \frac{3}{\nu} + \frac{\alpha \theta}{2} + \alpha(1-\theta) \left( \frac{2}{d} + \frac{3}{e}\right) + \frac{2}{k} +\frac{3}{l}.
\]
Since $(d,e), (k,l) \in S_{\gamc}$, it follows that
\[
\frac{3}{\nu} = b+ \left(2-b-\frac{\alpha}{2}\right) \theta >b
\]
for any $\theta \in (0,1)$ due to the fact $\alpha<4-2b$. We thus can choose $(m,n) \in S_{-\gamc}$ and $(d,e), (k,l) \in S_{\gamc}$ so that \eqref{non-est-1} holds (taking $(d,e) \equiv (k,l)$ and $d=e$ for instance).

We next estimate
\begin{align}
\||x|^{-b} |u|^\alpha u\|_{L^{m'}(I,L^{n'}(B^c))} &\leq \||x|^{-b}\|_{L^\nu(B^c)} \||u|^\alpha u\|_{L^{m'}(I,L^\rho)} \nonumber \\
&\lesssim \|u\|^\alpha_{L^{\alpha q}(I,L^{\alpha r})} \|u\|_{L^k(I,L^l)} \nonumber \\
&\lesssim \|u\|^{\alpha \theta}_{L^\infty(I,L^2)} \|u\|^{\alpha (1-\theta)}_{L^d(I,L^e)} \|u\|_{L^k(I,L^l)} \label{non-est-2}
\end{align}
provided $\nu, \rho, q,r \geq 1$ and $(d,e), (k,l) \in S_{\gamc}$ satisfying
\[
\frac{1}{n'} = \frac{1}{\nu} + \frac{1}{\rho}, \quad \frac{3}{\nu}<b, \quad \frac{1}{\rho} = \frac{1}{r} + \frac{1}{l}
\]
and
\[
\frac{1}{m'} = \frac{1}{q} + \frac{1}{k}, \quad \frac{1}{\alpha q} = \frac{\theta}{\infty} + \frac{1-\theta}{d}, \quad \frac{1}{\alpha r} = \frac{\theta}{2} + \frac{1-\theta}{e}.
\]
Arguing as above, we see that
\[
2+\frac{2-b}{\alpha} = \frac{3}{\nu} + \frac{3\alpha\theta}{2} + \alpha(1-\theta) \left(\frac{2}{d} + \frac{3}{e}\right) + \frac{2}{k} + \frac{3}{l} 
\]
or
\[
\frac{3}{\nu} = 2- \frac{3\alpha\theta}{2} - (2-b)(1-\theta) = b - \left(\frac{3\alpha}{2} - 2+b\right) \theta <b
\]
for any $\theta \in (0,1)$ since $\alpha>\frac{4-2b}{3}$. There thus exist $(m,n) \in S_{-\gamc}$ and $(d,e), (k,l) \in S_{\gamc}$ so that \eqref{non-est-2} holds.

We next prove the third estimate, the second one is similar. We bound 
	\begin{align*}
	\|\nabla(|x|^{-b} |u|^\alpha u)\|_{L^2(I, L^{\frac{6}{5}})} &\leq \||x|^{-b} \nabla(|u|^\alpha u)\|_{L^2(I, L^{\frac{6}{5}})} + \||x|^{-b-1} |u|^\alpha u \|_{L^2(I, L^{\frac{6}{5}})} \\
	&\leq \||x|^{-b} \nabla(|u|^\alpha u)\|_{L^2(I, L^{\frac{6}{5}}(B))} + \||x|^{-b} \nabla(|u|^\alpha u)\|_{L^2(I, L^{\frac{6}{5}}(B^c))} \\
	&\mathrel{\phantom{\leq}} + \||x|^{-b-1} |u|^\alpha u \|_{L^2(I, L^{\frac{6}{5}}(B))} + \||x|^{-b-1} |u|^\alpha u \|_{L^2(I, L^{\frac{6}{5}}(B^c))}.
	\end{align*}
	On $B$, we use H\"older's inequality to have
	\begin{align}
	\||x|^{-b} \nabla(|u|^\alpha u)\|_{L^2(I, L^{\frac{6}{5}}(B))} &\leq \||x|^{-b}\|_{L^{\nu_1} (B)} \|\nabla (|u|^\alpha u)\|_{L^2(I, L^{\rho_1})} \nonumber \\
	&\lesssim \|u\|^\alpha_{L^{\alpha q_1}(I, L^{\alpha r_1})} \|\nabla u\|_{L^{m_1}(I,L^{n_1})} \nonumber \\
	&\lesssim \|u\|^{\alpha\theta_1}_{L^\infty(I,L^6)} \|u\|^{\alpha(1-\theta_1)}_{L^{k_1}(I,L^{l_1})} \|\nabla u\|_{L^{m_1}(I,L^{n_1})} \label{non-est-B-1}
	\end{align}
	provided $\nu_1, \rho_1, q_1, r_1 \geq 1$, $(k_1,l_1) \in S_{\gamc}$ and $(m_1,n_1) \in S_0$ satisfying $2 \leq n_1<3$ and
	\[
	\frac{5}{6} = \frac{1}{\nu_1} +\frac{1}{\rho_1},  \quad \frac{1}{2} = \frac{1}{q_1} + \frac{1}{m_1}, \quad \frac{1}{\rho_1}=\frac{1}{r_1}+\frac{1}{n_1}
	\]
	and
	\[
	\frac{3}{\nu_1} >b,\quad \frac{1}{\alpha q_1} =\frac{\theta_1}{\infty} +\frac{1-\theta_1}{k_1}, \quad \frac{1}{\alpha r_1} =\frac{\theta_1}{6} +\frac{1-\theta_1}{l_1}.
	\]
	It follows that
	\[
	\frac{7}{2} = \frac{3}{\nu_1} + \frac{\alpha\theta_1}{2} + \alpha(1-\theta_1) \left(\frac{2}{k_1} + \frac{3}{l_1}\right) + \frac{2}{m_1} +\frac{3}{n_1}.
	\]
	Since $(k_1,l_1) \in S_{\gamc}$ and $(m_1,n_1) \in S_0$, we see that
	\[
	\frac{3}{\nu_1} = 2-\frac{\alpha\theta_1}{2} -(2-b)(1-\theta_1) = b + \left(2-b-\frac{\alpha}{2}\right)\theta_1.
	\]
	Since $\alpha <4-2b$, it follows that $\frac{3}{\nu_1}>b$ for all $\theta_1 \in (0,1)$.
	On the other hand, we also have 
	\[
	\frac{5}{2} =\frac{3}{\nu_1} +\frac{\alpha\theta_1}{2} + \frac{3\alpha(1-\theta_1)}{l_1} + \frac{3}{n_1} 
	\]
	hence
	\[
	\frac{3}{\nu_1} = \frac{5}{2} - \frac{3\alpha}{l_1} +\left(\frac{3\alpha}{l_1} - \frac{\alpha}{2} \right) \theta_1 - \frac{3}{n_1}.
	\]
	Since $\frac{3\alpha}{2-b} < l_1 < 6$, the right hand side is an increasing function on $\theta_1$. We take $\theta_1$ close to 1 so that
	\[
	\frac{3}{\nu_1} \sim \frac{5}{2} - \frac{\alpha}{2} -\frac{3}{n_1}.
	\]
	Taking $n_1=3^-$, we see that the condition $\frac{3}{\nu_1}>b$ requires $\alpha <3-2b$. There thus exist $(k_1, l_1) \in S_{\gamc}$ and $(m_1,n_1) \in S_0$ satisfying $2 \leq n_1<3$ so that \eqref{non-est-B-1} holds.
	
	We next estimate
	\begin{align}
	\||x|^{-b-1} |u|^\alpha u\|_{L^2(I, L^{\frac{6}{5}}(B))} &\leq \||x|^{-b-1} \|_{L^{\nu_1}(B)} \||u|^\alpha u\|_{L^2(I, L^{\rho_1})} \nonumber \\
	&\lesssim \|u\|^\alpha_{L^{\alpha q_1}(I, L^{\alpha r_1})} \|u\|_{L^{m_1}(I,L^{\sigma_1})} \nonumber \\
	&\lesssim \|u\|^{\alpha \theta_1}_{L^\infty(I,L^6)} \|u\|^{\alpha(1-\theta_1)}_{L^{k_1}(I,L^{l_1})} \|\nabla u\|_{L^{m_1}(I,L^{n_1})} \label{non-est-B-2}
	\end{align}
	provided $\nu_1, \rho_1, q_1, r_1 \geq 1$, $1 <\sigma_1<3$, $(k_1,l_1) \in S_{\gamc}$ and $(m_1,n_1)\in S_0$ satisfying $2 \leq n_1<3$ and 
	\[
	\frac{5}{6} =\frac{1}{\nu_1} + \frac{1}{\rho_1}, \quad \frac{1}{2} =\frac{1}{q_1}+\frac{1}{m_1}, \quad \frac{1}{\rho_1}=\frac{1}{r_1} +\frac{1}{\sigma_1}
	\]
	and
	\[
	\frac{3}{\nu_1} >b+1, \quad \frac{1}{\alpha q_1} =\frac{\theta_1}{\infty} +\frac{1-\theta_1}{k_1}, \quad \frac{1}{\alpha r_1} =\frac{\theta_1}{6}+\frac{1-\theta_1}{l_1}, \quad \frac{1}{\sigma_1}=\frac{1}{n_1}-\frac{1}{3}.
	\]
	It follows that
	\[
	\frac{9}{2}=\frac{3}{\nu_1} +\frac{\alpha\theta_1}{2} + \alpha(1-\theta_1) \left( \frac{2}{k_1} +\frac{3}{l_1}\right) + \frac{2}{m_1}+\frac{3}{n_1} \quad \text{or} \quad \frac{3}{\nu_1} = b+1 + \left(2-b-\frac{\alpha}{2}\right) \theta_1.
	\]
	This shows that $\frac{3}{\nu_1}>b+1$ for all $\theta_1 \in (0,1)$ since $\alpha <4-2b$. On the other hand, we have 
	\[
	\frac{7}{2} = \frac{3}{\nu_1} + \frac{3\alpha}{l_1} - \left( \frac{3\alpha}{l_1} - \frac{\alpha}{2} \right) \theta_1 +\frac{3}{n_1} \quad \text{or} \quad \frac{3}{\nu_1} = \frac{7}{2} - \frac{3\alpha}{l_1} + \left( \frac{3\alpha}{l_1} - \frac{\alpha}{2} \right) \theta_1 - \frac{3}{n_1}.
	\]
	Taking $\theta_1$ close to 1 and $n_1=3^-$, the condition $\frac{3}{\nu_1}>b+1$ also implies that $\alpha<3-2b$. Thus \eqref{non-est-B-2} holds with some $(k_1,l_1) \in S_{\gamc}$ and $(m_1,n_1) \in S_0$ satisfying $2 \leq n_1<3$.
	
	We now estimate the terms on $B^c$. By H\"older's inequality,
	\begin{align}
	\||x|^{-b} \nabla(|u|^\alpha  u)\|_{L^2(I,L^{\frac{6}{5}}(B^c))} &\leq \||x|^{-b} \|_{L^{\nu_2}(B^c)} \|\nabla (|u|^\alpha u)\|_{L^2(I,L^{\rho_2})} \nonumber \\
	&\lesssim \|u\|^\alpha_{L^{\alpha q_2} (I,L^{\alpha r_2})} \|\nabla u\|_{L^{m_2}(I,L^{n_2})} \nonumber \\
	&\lesssim \|u\|^{\alpha \theta_2}_{L^\infty(I,L^2)} \|u\|^{\alpha (1-\theta_2)}_{L^{k_2}(I,L^{l_2})} \|\nabla u\|_{L^{m_2}(I,L^{n_2})} \label{non-est-Bc-1}
	\end{align}
	provided $\nu_2, \rho_2, q_2, r_2 \geq 1$, $(k_2,l_2) \in S_{\gamc}$ and $(m_2,n_2) \in S_0$ satisfying $2 \leq n_2<3$ and
	\[
	\frac{5}{6}=\frac{1}{\nu_2}+\frac{1}{\rho_2}, \quad \frac{1}{2} =\frac{1}{q_2}+\frac{1}{m_2}, \quad \frac{1}{\rho_2}=\frac{1}{r_2}+\frac{1}{n_2}
	\]
	and
	\[
	\frac{3}{\nu_2} <b, \quad \frac{1}{\alpha q_2} =\frac{\theta_2}{\infty} +\frac{1-\theta_2}{k_2}, \quad \frac{1}{\alpha r_2} = \frac{\theta_2}{2}+\frac{1-\theta_2}{l_2}.
	\]
	We see that
	\[
	\frac{7}{2} =\frac{3}{\nu_2} + \frac{3\alpha\theta_2}{2} +\alpha(1-\theta_2)\left(\frac{2}{k_2}+\frac{3}{l_2} \right) +\frac{2}{m_2}+\frac{3}{n_2}
	\]
	which implies
	\[
	\frac{3}{\nu_2} = b- \left(\frac{3\alpha}{2} - 2+b \right) \theta_2 <b
	\]
	for all $\theta_2 \in (0,1)$ since $\alpha>\frac{4-2b}{3}$. On the other hand, we have
	\[
	\frac{5}{2} =\frac{3}{\nu_2} +\frac{3\alpha \theta_2}{2} +\frac{3\alpha(1-\theta_2)}{l_2}+\frac{3}{n_2} \quad \text{or} \quad \frac{3}{\nu_2} = \frac{5}{2} - \frac{3\alpha}{l_2} - \left(\frac{3\alpha}{2} - \frac{3\alpha}{l_2} \right)\theta_2 -\frac{3}{n_2}.
	\]
	Since $l_2 \geq \frac{3\alpha}{2-b}^+ >2$, the right hand right is a decreasing function on $\theta_2$. We choose $\theta_2$ close to 0 and get
	\[
	\frac{3}{\nu_2} \sim \frac{5}{2} - \frac{3\alpha}{l_2} - \frac{3}{n_2}.
	\]
	Taking $n_2=3^-$ and $l_2=\frac{6\alpha}{3-2b}^-$, we obtain
	\[
	\frac{3}{\nu_2} <\frac{5}{2} - \frac{3-2b}{2} -1 =b.
	\]
	Note that the condition $\alpha <3-2b$ ensures $\frac{3\alpha}{2-b}<\frac{6\alpha}{3-2b} <6$. Thus \eqref{non-est-Bc-1} holds for some $(k_2,l_2) \in S_{\gamc}$ and $(m_2, n_2) \in S_0$ satisfying $2 \leq n_2<3$. 
	
	Finally, we estimate
	\begin{align}
	\||x|^{-b-1} |u|^\alpha u\|_{L^2(I, L^{\frac{6}{5}}(B^c))} &\leq \||x|^{-b-1} \|_{L^{n_2}(B^c)} \||u|^\alpha u\|_{L^2(I, L^{\rho_2})} \nonumber \\
	&\lesssim \|u\|^\alpha_{L^{\alpha q_2}(I, L^{\alpha r_2})} \|u\|_{L^{m_2}(I,L^{\sigma_2})} \nonumber \\
	&\lesssim \|u\|^{\alpha \theta_2}_{L^\infty(I,L^2)} \|u\|^{\alpha(1-\theta_2)}_{L^{k_2}(I,L^{l_2})} \|\nabla u\|_{L^{m_2}(I,L^{n_2})}  \label{non-est-Bc-2}
	\end{align}
	provided $\nu_2, \rho_2, q_2, r_2 \geq 1$, $1 <\sigma_2<3$, $(k_2,l_2) \in S_{\gamc}$ and $(m_2,n_2) \in S$ satisfying $2 \leq n_2<3$ and 
	\[
	\frac{5}{6} =\frac{1}{\nu_2} + \frac{1}{\rho_2}, \quad \frac{1}{2} =\frac{1}{q_2}+\frac{1}{m_2}, \quad \frac{1}{\rho_2}=\frac{1}{r_2} +\frac{1}{\sigma_2}
	\]
	and
	\[
	\frac{3}{\nu_2} < b+1, \quad \frac{1}{\alpha q_2} =\frac{\theta_2}{\infty} +\frac{1-\theta_2}{k_2}, \quad \frac{1}{\alpha r_2} =\frac{\theta_2}{2}+\frac{1-\theta_2}{l_2}, \quad \frac{1}{\sigma_2}=\frac{1}{n_2}-\frac{1}{3}.
	\]
	It follows that
	\[
	\frac{9}{2}=\frac{3}{\nu_2} +\frac{3\alpha\theta_2}{2} + \alpha(1-\theta_2) \left( \frac{2}{k_2} +\frac{3}{l_2}\right) + \frac{2}{m_2}+\frac{3}{n_2} \quad \text{or} \quad \frac{3}{\nu_2} = b+1-\left(\frac{3\alpha}{2}-2+b\right) \theta_2 <b+1.
	\]
	We also have 
	\[
	\frac{7}{2}=\frac{3}{\nu_2} +\frac{3\alpha\theta_2}{2} +\frac{3\alpha(1-\theta_2)}{l_2} +\frac{3}{n_2}
	\]
	which, by choosing $\theta_2$ close to 0, implies $\frac{3}{\nu_2} \sim \frac{7}{2} -\frac{3\alpha}{l_2} -\frac{3}{n_2}$. Taking $n_2=3^-$ and $l_2=\frac{6\alpha}{3-2b}^-$, we see that $\frac{3}{\nu_2} < b+1$. There thus exist $(k_2,l_2) \in S_{\gamc}$ and $(m_2,n_2) \in S_0$ satisfying $2 \leq n_2<3$ such that \eqref{non-est-Bc-2} holds. The proof is complete by collecting \eqref{non-est-1}--\eqref{non-est-Bc-2}.
	\hfill $\Box$
	
	\subsection{Energy scattering}
	Let us start with the following small data global well-posedness. 
	\begin{lemma} [Small data global well-posedness] \label{lem-small-gwp}
		Let $0<b<1$ and $\frac{4-2b}{3}<\alpha<3-2b$. Let $V: \R^3 \rightarrow \R$ satisfy \eqref{ass-V-1} and \eqref{ass-V-2}. Let $T>0$ and suppose that $\|u(T)\|_{H^1} \leq A$. Then there exists $\varrho =\varrho(A)>0$ such that if 
		\[
		\|e^{-i(t-T)\Hc} u(T)\|_{S(\dot{H}^{\gamc},[T,+\infty))} <\varrho, 
		\]
		then there exists a unique global solution to \eqref{INLS-V} with initial data $u(T)$ satisfying
		\[
		\|u\|_{S(\dot{H}^{\gamc}, [T,+\infty))} \leq 2 \|e^{-i(t-T)\Hc} u(T)\|_{S(\dot{H}^{\gamc}, [T,+\infty))}
		\]
		and
		\[
		\|\scal{\nabla} u\|_{S(L^2,[T,+\infty))} \leq 2C\|u(T)\|_{H^1}.
		\]
	\end{lemma}
	
	\begin{proof}
		Consider
		\[
		X:= \left\{u \ : \ \|u\|_{S(\dot{H}^{\gamc},I)} \leq M, \quad \|\scal{\nabla} u\|_{S(L^2,I)} \leq N \right\}
		\]
		equipped with the distance
		\[
		d(u,v) := \|u-v\|_{S(L^2,I)} +\|u-v\|_{S(\dot{H}^{\gamc},I)},
		\]
		where $I=[T,+\infty)$ and $M, N>0$ will be chosen later. We will show that the functional
		\[
		\Phi(u(t)):= e^{-i(t-T) \Hc} u(T) \mp i \int_T^t e^{-i(t-s)\Hc} |x|^{-b} |u|^\alpha u(s) ds
		\]
		is a contraction on $(X,d)$. By Strichartz estimates (see e.g. \cite{HR}) and Lemma $\ref{lem-non-est}$, 
		\begin{align*}
		\|\Phi(u)\|_{S(\dot{H}^{\gamc}, I)} &\leq \|e^{-i(t-T)\Hc} u(T)\|_{S(\dot{H}^{\gamc},I)} + \left\| \int_T^t e^{-i(t-s)\Hc} |x|^{-b} |u|^\alpha u(s) ds \right\|_{S(\dot{H}^{\gamc},I)} \\
		&\leq \|e^{-i(t-T)\Hc} u(T)\|_{S(\dot{H}^{\gamc},I)} + C \||x|^{-b} |u|^\alpha u\|_{S'(\dot{H}^{-\gamc},I)} \\
		&\leq \|e^{-i(t-T)\Hc} u(T)\|_{S(\dot{H}^{\gamc},I)} + C \|u\|^{\alpha \theta}_{L^\infty(I,H^1)} \|u\|^{1+\alpha(1-\theta)}_{S(\dot{H}^{\gamc},I)}
		\end{align*}
		for some $\theta \in (0,1)$. Similarly, by Lemma $\ref{lem-non-est}$ and the equivalence of Sobolev norms, 
		\begin{align*}
		\|\scal{\nabla} \Phi(u)\|_{S(L^2,I)} &\leq \|\scal{\nabla} e^{-i(t-T)\Hc} u(T)\|_{S(L^2,I)} + \left\| \scal{\nabla} \int_T^t e^{-i(t-s)\Hc} |x|^{-b} |u|^\alpha u(s) ds \right\|_{S(L^2,I)} \\
		&\leq C\|\scal{\Lambda} u(T)\|_{L^2} + \|\scal{\Lambda} (|u|^\alpha u)\|_{L^2(I,L^{\frac{6}{5}})} \\
		&\leq C \|u(T)\|_{H^1} + \|\scal{\nabla} (|x|^{-b} |u|^\alpha u)\|_{L^2(I,L^{\frac{6}{5}})} \\
		&\leq C\|u(T)\|_{H^1} + \Big( \|u\|^{\alpha \theta_1}_{L^\infty(I,H^1)} \|u\|^{\alpha(1-\theta_1)}_{S(\dot{H}^{\gamc},I)} + \|u\|^{\alpha \theta_2}_{L^\infty(I,H^1)} \|u\|^{\alpha(1-\theta_2)}_{S(\dot{H}^{\gamc},I)} \Big) \|\scal{\nabla} u\|_{S(L^2,I)}
		\end{align*}
		for some $\theta_1,\theta_2 \in (0,1)$. We also have
		\begin{align*}
		\|\Phi(u) - \Phi(v)\|_{S(\dot{H}^{\gamc},I)} &\leq C \||x|^{-b} (|u|^\alpha u - |v|^\alpha v)\|_{S'(\dot{H}^{-\gamc},I)} \\
		&\leq C \left(\|u\|^{\alpha \theta}_{L^\infty(I,H^1)} \|u\|^{\alpha (1-\theta)}_{S(\dot{H}^{\gamc},I)} + \|v\|^{\alpha \theta}_{L^\infty(I,H^1)} \|v\|^{\alpha (1-\theta)}_{S(\dot{H}^{\gamc},I)} \right) \|u-v\|_{S(\dot{H}^{\gamc},I)}
	\end{align*}
	and
	\begin{align*}
	\|\Phi(u) - \Phi(v)\|_{S(L^2,I)} &\leq \left\|\int_T^t e^{-i(t-s)\Hc} |x|^{-b} ( |u|^\alpha u - |v|^\alpha v)(s) ds \right\|_{S(L^2,I)} \\
	&\leq C\||x|^{-b} (|u|^\alpha u - |v|^\alpha v) \|_{L^2(I,L^{\frac{6}{5}})} \\
	&\leq C \Big( \|u\|^{\alpha \theta_1}_{L^\infty(I,H^1)} \|u\|^{\alpha(1-\theta_1)}_{S(\dot{H}^{\gamc},I)} + \|u\|^{\alpha \theta_2}_{L^\infty(I,H^1)} \|u\|^{\alpha(1-\theta_2)}_{S(\dot{H}^{\gamc},I)} \\
	&\mathrel{\phantom{\leq C \Big( }} + \|v\|^{\alpha \theta_1}_{L^\infty(I,H^1)} \|v\|^{\alpha(1-\theta_1)}_{S(\dot{H}^{\gamc},I)} + \|v\|^{\alpha \theta_2}_{L^\infty(I,H^1)} \|v\|^{\alpha(1-\theta_2)}_{S(\dot{H}^{\gamc},I)} \Big) \|u-v\|_{S(L^2,I)}.
	\end{align*}
	There thus exists $C>0$ independent of $T$ such that for any $u,v \in X$, 
	\begin{align*}
	\|\Phi(u)\|_{S(\dot{H}^{\gamc},I)} &\leq \|e^{-i(t-T)\Hc}  u(T)\|_{S(\dot{H}^{\gamc},I)} + C N^{\alpha \theta} M^{1+\alpha(1-\theta)}, \\
	\|\scal{\nabla} \Phi(u)\|_{S(L^2,I)} &\leq C\|u(T)\|_{H^1} + C \left(N^{\alpha\theta_1} M^{\alpha(1-\theta_1)} + N^{\alpha \theta_2} M^{\alpha(1-\theta_2)} \right) N
	\end{align*}
	and
	\begin{align*}
	d(\Phi(u),\Phi(v)) \leq C \left(N^{\alpha \theta_1} M^{\alpha(1-\theta_1)} + N^{\alpha \theta_2} M^{\alpha(1-\theta_2)} \right) d(u,v).
	\end{align*}
	We now choose $M= 2 \|e^{-i(t-T)\Hc} u(T)\|_{S(\dot{H}^{\gamc},I)}$ and $N=2C\|u(T)\|_{H^1}$. By taking $M$ sufficiently small so that
	\[
	CN^{\alpha \theta} M^{1+\alpha(1-\theta)} \leq \frac{M}{2}, \quad C \left(N^{\alpha\theta_1} M^{\alpha(1-\theta_1)} + N^{\alpha\theta_2} M^{\alpha(1-\theta_2)} \right) \leq \frac{1}{2},
	\] 
	we see that $\Phi$ is a contraction mapping on $(X,d)$. The proof is complete.
	\end{proof}
	
	\begin{lemma}[Small data scattering] \label{lem-small-scat}
		Let $0<b<1$ and $\frac{4-2b}{3}<\alpha <3-2b$. Let $V:\R^3 \rightarrow \R$ satisfy \eqref{ass-V-1} and \eqref{ass-V-2}. Let $u$ be a global solution to \eqref{INLS-V} satisfying 
		\[
		\|u\|_{L^\infty(\R, H^1)} \leq A.
		\]
		Then there exists $\varrho=\varrho(A)>0$ such that if 
		\[
		\|e^{-i(t-T)\Hc} u(T)\|_{S(\dot{H}^{\gamc},[T,+\infty))} <\varrho
		\]
		for some $T>0$, then $u$ scatters in $H^1$ forward in time.
	\end{lemma}
	
	\begin{proof}
		Let $\varrho=\varrho(A)$ be as in Lemma $\ref{lem-small-gwp}$. It follows from Lemma $\ref{lem-small-gwp}$ that the solution satisfies
		\begin{align*}
		\|u\|_{S(\dot{H}^{\gamc},[T,+\infty))} &\leq 2\|e^{-i(t-T)\Hc} u(T)\|_{S(\dot{H}^{\gamc},[T,+\infty))}, \\
		\|\scal{\nabla} u\|_{S(L^2,[T,+\infty))} &\leq 2 C\|u(T)\|_{H^1}.
		\end{align*}
		Now let $0<\tau<t<+\infty$. By Strichartz estimates, Lemma $\ref{lem-non-est}$ and the equivalence of Sobolev norms,
		\begin{align*}
		\|e^{it\Hc} u(t) -e^{i\tau\Hc} u(\tau)\|_{H^1} & = \left\| \int_\tau^t e^{is\Hc} |x|^{-b} |u|^\alpha u(s) ds \right\|_{H^1} \\
		&\leq C \|\scal{\Lambda} (|x|^{-b} |u|^\alpha u)\|_{L^2((\tau,t),L^{\frac{6}{5}})} \\
		&\leq C \|\scal{\nabla} (|x|^{-b} |u|^\alpha u)\|_{L^2((\tau,t),L^{\frac{6}{5}})} \\
		&\leq C \Big(\|u\|^{\alpha \theta_1}_{L^\infty((\tau,t),H^1)} \|u\|^{\alpha(1-\theta_1)}_{S(\dot{H}^{\gamc},(\tau,t))}   \\
		&\mathrel{\phantom{\leq C Big( }} + \|u\|^{\alpha \theta_2}_{L^\infty((\tau,t),H^1)} \|u\|^{\alpha(1-\theta_2)}_{S(\dot{H}^{\gamc},(\tau,t))} \Big) \|\scal{\nabla} u\|_{S(L^2,(\tau,t))} \rightarrow 0
		\end{align*}
		as $\tau, t \rightarrow +\infty$. This shows that $(e^{it\Hc} u(t))_t$ is a Cauchy sequence in $H^1$. Thus the limit
		\[
		u_0^+:= u_0 + i \int_t^\infty e^{is\Hc} |x|^{-b} |u|^\alpha u(s) ds
		\]
		exists in $H^1$. Arguing as above, we prove as well that
		\[
		\|u(t) - e^{-it\Hc} u_0^+\|_{H^1} \rightarrow 0
		\]
		as $t\rightarrow +\infty$. The proof is complete.
	\end{proof}
	
	\begin{proposition} \label{prop-scat-focus}
		Let $0<b<1$ and $\frac{4-2b}{3} <\alpha<3-2b$. Let $V: \R^3 \rightarrow \R$ satisfy \eqref{ass-V-1}, $V\geq 0, x\cdot \nabla V \leq 0, x \cdot \nabla V \in L^{\frac{3}{2}}$ and $V$ be radially symmetric. Let $u_0 \in H^1$ be radially symmetric and satisfy \eqref{cond-ener} and \eqref{cond-grad-glob}. Then for any $\vareps>0$, there exists $T=T(\vareps, u_0,Q)$ sufficiently large such that the corresponding global solution to the focusing problem \eqref{INLS-V} satisfies
		\[
		\|e^{-i(t-T)\Hc} u(T)\|_{S(\dot{H}^{\gamc},[T,+\infty))} \lesssim \vareps^\nu
		\]
		for some $\nu>0$. 
	\end{proposition}

\begin{proof}
	The proof is divided into several steps.

	\noindent {\bf Step 1. Estimate the linear part.} By Strichartz estimates (see e.g. \cite{HR}),
	\[
	\|e^{-it\Hc} u_0\|_{S(\dot{H}^{\gamc},\R)} \lesssim \|u_0\|_{\dot{H}^{\gamc}} \lesssim \|u_0\|_{H^1} \leq C(u_0,Q)<\infty.
	\]
	By the monotone convergence theorem and the fact that $\left(\infty,\frac{3\alpha}{2-b}\right)$ does not belong to $S_{\gamc}$, we may find $T>\vareps^{-\sigma}$ (with some $\sigma>0$ to be chosen later) depending on $u_0$ and $Q$ so that
	\begin{align} \label{linear-small}
	\|e^{-it\Hc} u_0\|_{S(\dot{H}^{\gamc},[T,+\infty))} \lesssim \vareps.
	\end{align}
	
	\noindent {\bf Step 2. Estimate the nonlinear part.} By enlarging $T$ if necessary, we have from \eqref{small-L2} that for any $R>0$,
	\[
	\int_{|x| \leq R} |u(T,x)|^2 dx \lesssim \vareps.
	\]
	By the definition of $\chi_R$,
	\[
	\int \chi_R(x) |u(T,x)|^2 dx \lesssim \vareps.
	\]
	Using the fact
	\begin{align*}
	\left| \frac{d}{dt} \int \chi_R(x) |u(t,x)|^2 dx\right| &= \left| 2\int \nabla \chi_R(x) \cdot \ima (\overline{u}(t,x) \nabla u(t,x)) dx \right| \\
	&\leq 2 \|\nabla \chi_R\|_{L^\infty} \|u(t)\|_{L^2} \|\nabla u(t)\|_{L^2} \\
	&\lesssim R^{-1}
	\end{align*}
	for all $t\in \R$. It follows that for any $t\leq T$,
	\begin{align*}
	\int \chi_R(x) |u(t,x)|^2 dx &= \int \chi_R(x)|u(T,x)|^2 dx - \int_t^T \left( \frac{d}{ds} \int \chi_R(x) |u(s,x)|^2 dx \right) ds \\
	&\leq \int \chi_R(x) |u(T,x)|^2 dx + CR^{-1} (T-t)
	\end{align*}
	for some constant $C=C(u_0,Q)>0$. By choosing $R>\vareps^{-1-\sigma}$ with some $\sigma>0$ to be chosen later, we see that for all $t\in I:= [T-\vareps^{-\sigma},T]$,
	\[
	\int \chi_R(x)|u(t,x)|^2 dx \leq C\vareps + C R^{-1} \vareps^{-\sigma} \leq 2C\vareps
	\]
	hence
	\[
	\|\chi_R^{\frac{1}{2}} u\|_{L^\infty(I,L^2)} \lesssim \vareps^{\frac{1}{2}}.
	\]
	Now let $(k,l) \in S_{\gamc}$. By H\"older's inequality and the radial Sobolev embedding,
	\begin{align*}
	\|u\|_{L^\infty(I,L^l)} &\leq \|\chi_R^{\frac{1}{2}} u\|_{L^\infty(I,L^l)} + \|(1-\chi_R^{\frac{1}{2}}) u\|_{L^\infty(I,L^l)} \\
	&\lesssim \|\chi_R^{\frac{1}{2}} u\|^{\frac{6-l}{2l}}_{L^\infty(I,L^2)} \|\chi_R^{\frac{1}{2}} u\|^{\frac{3l-6}{2l}}_{L^\infty(I,L^6)} + \|(1-\chi_R^{\frac{1}{2}}) u\|^{\frac{2}{l}}_{L^\infty(I,L^2)} \|(1-\chi_R^{\frac{1}{2}}) u\|^{\frac{l-2}{l}}_{L^\infty(I,L^\infty)}  \\
	&\lesssim \vareps^{\frac{6-l}{l}} + R^{-\frac{l-2}{l}} \\
	&\lesssim \vareps^{\frac{6-l}{l}}
	\end{align*}
	provided $R>\vareps^{-\frac{6-l}{l-2}}$. We next use the Duhamel formula to write
	\begin{align} \label{duhamel}
	e^{-i(t-T)\Hc} u(T) =e^{-it\Hc}u_0 + i \int_0^T e^{-i(t-s)\Hc} |x|^{-b} |u|^\alpha u(s) ds = e^{-it\Hc} u_0 + F_1(t) + F_2(t),
	\end{align}
	where
	\[
	F_1(t) := i\int_I e^{-i(t-s)\Hc} |x|^{-b} |u|^\alpha u(s) ds, \quad F_2(t):= i \int_J e^{-i(t-s)\Hc} |x|^{-b} |u|^\alpha u(s) ds
	\]
	with $I$ as above and $J:= [0,T-\vareps^{-\sigma}]$. By Strichartz estimates and Lemma $\ref{lem-non-est}$,
	\[
	\|F_1\|_{S(\dot{H}^{\gamc},[T,+\infty))} \lesssim \||x|^{-b} |u|^\alpha u\|_{S(\dot{H}^{-\gamc},I)} \lesssim \|u\|^{\alpha \theta}_{L^\infty(I,H^1)} \|u\|^{1+\alpha(1-\theta)}_{S(\dot{H}^{\gamc},I)}
	\]
	for some $\theta \in (0,1)$. By the definition of $S_{\gamc}$ (see \eqref{def-gamc-adm}) and the fact $\frac{4-2b}{3}<\alpha<3-2b$, there exists $\vartheta >0$ small depending on $\alpha$ and $b$ such that
	\begin{align} \label{choice-vartheta}
	2+\vartheta \leq k, \quad \frac{3\alpha}{2-b}+\vartheta \leq l \leq 6-\vartheta
	\end{align}
	for any $(k,l) \in S_{\gamc}$. We estimate
	\begin{align*}
	\|u\|_{S(\dot{H}^{\gamc},I)} &= \sup_{(k,l) \in S_{\gamc}} \|u\|_{L^k(I,L^l)} \\
	&\lesssim \sup_{(k,l) \in S_{\gamc}} |I|^{\frac{1}{k}} \|u\|_{L^\infty(I,L^l)} \\
	&\lesssim \sup_{(k,l) \in S_{\gamc}} \vareps^{\frac{6-l}{l} - \frac{\sigma}{k}} \\
	&\lesssim \vareps^{\frac{\vartheta}{6-\vartheta} - \frac{\sigma}{2+\vartheta}}.
	\end{align*}
	This implies that
	\begin{align} \label{est-F1}
	\|F_1\|_{S(\dot{H}^{\gamc},[T,\infty))} \lesssim \vareps^{\left(\frac{\vartheta}{6-\vartheta} - \frac{\sigma}{2+\vartheta}\right)(1+\alpha(1-\theta))}.
	\end{align}
	To estimate $F_2$, we observe that for each $(k,l) \in S_{\gamc}$, there exists $\eta \in (0,1)$ such that
	\[
	\frac{1}{k} =\frac{\eta}{d} + \frac{1-\eta}{m}, \quad \frac{1}{l} =\frac{\eta}{e}
	\]
	for some $(d,e) \in S_0$ and $m>2$. In fact, the condition $\frac{2}{d}+\frac{3}{e}=\frac{3}{2}$ implies
	\begin{align} \label{choice-eta}
	\frac{3}{2} -\gamc = \frac{3\eta}{2} + \frac{3(1-\eta)}{m}.
	\end{align}
	To ensure $(d,e) \in S_0$, we also need $e \in [2,6]$. Using \eqref{choice-vartheta}, we see that
	\[
	e = l \eta \in \left( \left(\frac{3\alpha}{2-b}+\vartheta\right) \eta , (6-\vartheta)\eta\right). 
	\]
	Since $\alpha>\frac{4-2b}{3}$, the condition $e \in [2,6]$ is satisfied by taking $\eta=1-\epsilon$ with $0<\epsilon \ll 1$. With this choice, $\eqref{choice-eta}$ becomes
	\[
	\frac{2-b}{\alpha} = \frac{3(1-\epsilon)}{2} + \frac{3\epsilon}{m} \quad \text{or} \quad m = \frac{6\alpha \epsilon}{2(2-b) - 3\alpha(1-\epsilon)}.
	\]
	It is easy to check that $m>2$ since $\alpha>\frac{4-2b}{3}$. We thus estimate
	\[
	\|F_2\|_{L^k([T,+\infty),L^l)} \leq \|F_2\|^\eta_{L^d([T,+\infty),L^e)} \|F_2\|^{1-\eta}_{L^m([T,+\infty),L^\infty)}.
	\]
	Using the fact
	\[
	F_2(t) = e^{-i(t-T+\vareps^{-\sigma})\Hc} u(T-\vareps^{-\sigma}) - e^{-it\Hc} u_0,
	\]
	we have
	\[
	\|F_2\|_{L^d([T,+\infty),L^e)} \lesssim 1.
	\]
	On the other hand, by the dispersive estimate \eqref{disper-est}, 
	\begin{align*}
	\|F_2(t)\|_{L^\infty} &\leq \left\| \int_J e^{-i(t-s)\Hc} |x|^{-b} |u|^\alpha u(s) ds \right\|_{L^\infty} \\
	&\lesssim \int_J |t-s|^{-\frac{3}{2}} \left( \int |x|^{-b} |u(s,x)|^{\alpha+1} dx\right) ds \\
	&\lesssim \int_0^{T-\vareps^{-\sigma}} |t-s|^{-\frac{3}{2}} \|u(s)\|^{\alpha+1}_{H^1} ds \\
	&\lesssim   \int_0^{T-\vareps^{-\sigma}} |t-s|^{-\frac{3}{2}}  ds \\
	&\lesssim (t-T+\vareps^{-\sigma})^{-\frac{1}{2}}.
	\end{align*}
	Here we have used the fact
	\[
	\int |x|^{-b} |u(s,x)|^{\alpha+1} dx \lesssim \|u(s)\|^{\alpha+1}_{H^1}
	\]
	which follows from the Sobolev embedding and the fact $\frac{4-2b}{3}<\alpha<3-2b$. We thus get
	\[
	\|F_2\|_{L^m([T,+\infty),L^\infty)} \lesssim \left( \int_T^\infty (t-T+\vareps^{-\sigma})^{-\frac{m}{2}} dt \right)^{\frac{1}{m}} \lesssim \vareps^{\frac{(m-2)\sigma}{2m}}
	\]
	hence
	\[
	\|F_2\|_{L^k([T,+\infty),L^l)} \lesssim \vareps^{\frac{(m-2)\sigma \epsilon}{2m}} =\vareps^{\frac{(3\alpha-4+2b)\sigma}{6\alpha}}.
	\]
	This shows that
	\begin{align} \label{est-F2}
	\|F_2\|_{S(\dot{H}^{\gamc},[T,+\infty))} \lesssim  \vareps^{\frac{(3\alpha-4+2b)\sigma}{6\alpha}}.
	\end{align}
	
	\noindent {\bf Step 3. Conclusion.} Collecting \eqref{duhamel}, \eqref{linear-small}, \eqref{est-F1} and \eqref{est-F2}, we show that
	\[
	\|e^{-i(t-T)\Hc} u(T)\|_{S(\dot{H}^{\gamc},[T,+\infty))} \lesssim \vareps^\nu
	\]
	for some $\nu>0$ provided $\sigma>0$ is chosen sufficiently small. The proof is complete.
\end{proof}
	
	\begin{remark} \label{rem-scat-refi}
		Using Remark $\ref{rem-coro-mora-est-refi}$, Proposition $\ref{prop-scat-focus}$ still holds if we assume \eqref{cond-grad-glob-refi} in place of \eqref{cond-grad-glob}. 
	\end{remark}

	\noindent {\bf Proof of the scattering part given in Theorem $\ref{theo-dyna-focus}$.}
	It follows immediately from Lemma $\ref{lem-small-gwp}$, Lemma $\ref{lem-small-scat}$ and Proposition $\ref{prop-scat-focus}$.
	\hfill $\Box$
	
	\begin{remark} \label{rem-glob-refi}
		Using Remark $\ref{rem-scat-refi}$, the energy scattering still holds if we assume \eqref{cond-grad-glob-refi} instead of \eqref{cond-grad-glob}. 
	\end{remark}

	\noindent {\bf Proof of Theorem $\ref{theo-scat-defocus}$.}
	The proof of this result is similar to Proposition $\ref{prop-scat-focus}$ by using \eqref{small-L2-defocus} instead of \eqref{small-L2}. We thus omit the details. 
	\hfill $\Box$
	
	\section{Blow-up criteria}
	\label{S4}
	\setcounter{equation}{0}
	In this section, we give give the proof of the blow-up part given in Theorem $\ref{theo-dyna-focus}$. Let us start with the following blow-up criteria.
	\begin{proposition}[Blow-up criteria] \label{prop-blow-crite}
		Let $0<b<1$ and $\frac{4-2b}{3}<\alpha<4-2b$. Let $V: \R^3 \rightarrow \R$ satisfy \eqref{ass-V-1}, $V\geq 0, x \cdot \nabla V \in L^{\frac{3}{2}}$, $2V + x\cdot \nabla V \geq 0$. Let $u:[0,T^*) \times \R^3 \rightarrow \C$ be a $H^1$ maximal solution to the focusing problem \eqref{INLS-V}. Assume that there exists $\delta>0$ such that 
		\begin{align} \label{ass-K}
		\sup_{t\in[0,T^*)} K(u(t)) \leq -\delta,
		\end{align}
		where $K(u(t))$ is as in \eqref{defi-K}. Then either $T^*<+\infty$ or $T^*=+\infty$ and there exists a time sequence $t_n \rightarrow +\infty$ such that 
		\[
		\lim_{n\rightarrow \infty} \|\nabla u(t_n)\|_{L^2} =\infty.
		\]
	\end{proposition}
		
	Before giving the proof of this result, let us prove the blow-up part given in Theorem $\ref{theo-dyna-focus}$.
		
	\noindent {\bf Proof of the blow-up part given in Theorem $\ref{theo-dyna-focus}$.}
		By Proposition $\ref{prop-blow-crite}$, it suffices to show \eqref{ass-K} for some $\delta>0$. Multiplying $K(u(t))$ with $[M(u(t))]^{\sigc}$ and using the assumption $2V + x\cdot \nabla V \geq 0$, we have
		\begin{align*}
		K(u(t)) [M(u(t))]^{\sigc} &= \left( \|\nabla u(t)\|^2_{L^2} - \frac{1}{2} \int x \cdot \nabla V |u(t)|^2 dx - \frac{3\alpha+2b}{2(\alpha+2)} \int |x|^{-b} |u(t)|^{\alpha+2} dx \right) \|u(t)\|^{2\sigc}_{L^2} \\
		&\leq \left( \|\nabla u(t)\|^2_{L^2} + \int V |u(t)|^2 dx - \frac{3\alpha+2b}{2(\alpha+2)} \int |x|^{-b} |u(t)|^{\alpha+2} dx  \right) \|u(t)\|^{2\sigc}_{L^2} \\
		&= \left( \|\Lambda u(t)\|^2_{L^2} - \frac{3\alpha+2b}{2(\alpha+2)} \int |x|^{-b} |u(t)|^{\alpha+2} dx \right) \|u(t)\|^{2\sigc}_{L^2} \\
		&= \frac{3\alpha+2b}{2} E(u(t)) [M(u(t))]^{\sigc} - \frac{3\alpha-4+2b}{4} \left( \|\Lambda u(t)\|_{L^2} \|u(t)\|^{\sigc}_{L^2} \right)^2 \\
		&= \frac{3\alpha+2b}{2} E(u_0) [M(u_0)]^{\sigc} - \frac{3\alpha-4+2b}{4} \left( \|\Lambda u(t)\|_{L^2} \|u(t)\|^{\sigc}_{L^2} \right)^2
		\end{align*}
		for all $t$ in the existence time. By \eqref{cond-ener}, there exists $\theta=\theta(u_0,Q)>0$ such that
		\[
		E(u_0) [M(u_0)]^{\sigc} < (1-\theta) E_0(Q) [M(Q)]^{\sigc}.
		\]
		By \eqref{est-solu-blow} and the fact
		\[
		E_0(Q) [M(Q)]^{\sigc} = \frac{3\alpha-4+2b}{2(3\alpha+2b)} \left(\|\nabla Q\|_{L^2} \|Q\|^{\sigc}_{L^2} \right)^2,
		\]
		we infer that
		\begin{align*}
		K(u(t))[M(u(t))]^{\sigc} \leq -\frac{(3\alpha-4+2b)\theta}{4} \left(\|\nabla Q\|_{L^2} \|Q\|^{\sigc}_{L^2} \right)^2
		\end{align*}
		for all $t$ in the existence time. It follows that
		\[
		K(u(t)) \leq -\frac{(3\alpha-4+2b)\theta}{4}  \|\nabla Q\|^2_{L^2} \left(\frac{M(Q)}{M(u_0)}\right)^{\sigc} =:-\delta
		\]
		for all $t$ in the existence time which proves \eqref{ass-K}.
		
		Now assume in addition to \eqref{cond-blow} that $u_0 \in L^2(|x|^2dx)$. It is well-known that the corresponding solution belongs to $L^2(|x|^2 dx))$ for all $t \in [0,T^*)$ and by \eqref{viri-iden},
		\[
		\frac{d^2}{dt^2} \|xu(t)\|^2_{L^2} = 8 K(u(t))
		\] 
		for all $t\in [0,T^*)$. It follows that
		\[
		\frac{d^2}{dt^2} \|xu(t)\|^2_{L^2} \leq -8\delta <0, \quad \forall t\in [0,T^*).
		\]
		By the classical argument of Glassey \cite{Glassey}, the corresponding solution must blow-up in finite time. The proof is complete.
	\hfill $\Box$

	Before giving the proof of Proposition $\ref{prop-blow-crite}$, we need the following lemmas.
	
	\begin{lemma} [$L^2$-estimate outside a large ball]  \label{lem-L2-outside-ball}
		Let $0<b<1$ and $0<\alpha<4-2b$. Let $V: \R^3 \rightarrow \R$ satisfy \eqref{ass-V-1} and \eqref{ass-V-2}. Assume that $u\in C([0,+\infty),H^1)$ is a solution to \eqref{INLS-V} satisfying 
		\[
		\sup_{t\in [0,+\infty)} \|\nabla u(t)\|_{L^2} <\infty.
		\]
		Then there exists $C>0$ such that for any $\vareps>0$ and any $R>0$,
		\[
		\int_{|x|>R} |u(t,x)|^2 dx \leq o_R(1) +\vareps
		\]
		for any $t\in \left[0,\frac{\vareps R}{C}\right]$.
	\end{lemma}
	
	\begin{proof}
		Let $\vartheta: [0,\infty) \rightarrow [0,1]$ be a smooth function satisfying
		\[
		\vartheta(r)= \left\{
		\begin{array}{cl}
		0 &\text{if } 0 \leq r \leq \frac{1}{2}, \\
		1 &\text{if } r\geq 1.
		\end{array}
		\right.
		\]
		Given $R>0$, we denote the radial function
		\[
		\psi_R(x)= \psi_R(r):= \vartheta(r/R), \quad r=|x|.
		\]
		It follows that $\nabla \psi_R(x) = \frac{x}{rR} \vartheta'(r/R)$ and $\|\psi_R\|_{L^\infty} \lesssim R^{-1}$.  We define
		\begin{align} \label{def-V-psi-R}
		V_{\psi_R}(t):= \int \psi_R |u(t)|^2 dx.
		\end{align}
		By the fundamental theorem of calculus,
		\begin{align*}
		V_{\psi_R}(t) &= V_{\psi_R}(0) + \int_0^t \frac{d}{ds} V_{\psi_R}(s) ds \\
		&\leq V_{\psi_R}(0) + \left(\sup_{s\in[0,t]} \left|\frac{d}{ds} V_{\psi_R}(s) \right|  \right) t \\
		&= V_{\psi_R}(0) + \left( \sup_{s\in[0,t]} \left|2 \int \nabla \psi_R \cdot \ima (\overline{u}(s) \nabla u(s)) dx \right|\right) t \\
		&\leq V_{\psi_R}(0) + 2 \|\nabla \psi_R\|_{L^\infty} \left(\sup_{s \in [0,t]} \|u(s)\|_{L^2} \|\nabla u(s)\|_{L^2} \right) t \\
		&\leq V_{\psi_R}(0) + \frac{C}{R} t.
		\end{align*}
		By the choice of $\vartheta$, we see that
		\[
		V_{\psi_R}(0) = \int \psi_R(x) |u_0(x)|^2 dx  \leq \int_{|x|>R/2} |u_0(x)|^2 dx  \rightarrow 0
		\]
		as $R\rightarrow \infty$, hence $V_{\psi_R}(0) = o_R(1)$. The result follows by using the fact 
		\[
		\int_{|x|\geq R} |u(t,x)|^2 dx \leq V_{\psi_R}(t).
		\]
		The proof is complete. 
	\end{proof}

	Now let $\varphi_R$ be as in \eqref{def-varphi-R} and denote $V_{\varphi_R}(t)$ as in \eqref{def-V-psi-R} with $\varphi_R$ instead of $\psi_R$. We have the following localized virial estimates.
	\begin{lemma} \label{lem-local-viri-est}
		Let $0<b<1$ and $0<\alpha<4-2b$. Let $V:\R^3 \rightarrow \R$ satisfy \eqref{ass-V-1}, \eqref{ass-V-2} and $x \cdot \nabla V \in L^{\frac{3}{2}}$. Let $u_0 \in H^1$ and $u:[0,T^*) \times \R^3 \rightarrow \C$ be the corresponding solution to the focusing problem \eqref{INLS-V}. Then it holds that for any $R>0$,
		\[
		\frac{d^2}{dt^2} V_{\varphi_R}(t) \leq 8 K(u(t)) + o_R(1) \|u(t)\|^2_{H^1} + CR^{-b} \|u(t)\|^{\frac{4-\alpha}{2}}_{L^2(|x|>R)} \|u(t)\|^{\frac{3\alpha}{2}}_{H^1}
		\]
		for some constant $C>0$ independent of $R$.
	\end{lemma}

	\begin{proof}
		Using the fact $\frac{d}{dt} V_{\varphi_R}(t) = M_{\varphi_R}(t)$, we have from Lemma $\ref{lem-virial-iden}$ and Remark $\ref{rem-virial-iden}$ that
		\begin{align*}
		\frac{d^2}{dt^2} V_{\varphi_R}(t) = &- \int \Delta^2 \varphi_R |u(t)|^2 dx + 4 \int \frac{\varphi'_R(r)}{r} |\nabla u(t)|^2 dx \\
		&+ 4\int \left( \frac{\varphi''_R(r)}{r^2} - \frac{\varphi'_R(r)}{r^3}\right) |x \cdot \nabla u(t)|^2 dx - 2\int \frac{\varphi'_R(r)}{r} x \cdot \nabla V |u(t)|^2 dx \\
		& - \frac{2\alpha}{\alpha+2} \int |x|^{-b} \Delta \varphi_R |u(t)|^{\alpha+2} dx - \frac{4b}{\alpha+2} \int |x|^{-b} \frac{\varphi'_R(r)}{r} |u(t)|^{\alpha+2} dx.
		\end{align*}
		Since $\varphi_R(r) = r^2$ on $r=|x| \leq R$, we see that
		\begin{align*}
		\frac{d^2}{dt^2} V_{\varphi_R}(t) &= 8 K(u(t)) - 8\|\nabla u(t)\|^2_{L^2(|x|>R)} \\
		&\mathrel{\phantom{=}} + 4 \int_{|x|>R} x\cdot \nabla V |u(t)|^2 dx + \frac{4(3\alpha+2b)}{\alpha+2} \int_{|x|>R} |x|^{-b} |u(t)|^{\alpha+2} dx \\
		&\mathrel{\phantom{=}} - \int_{|x|>R} \Delta^2 \varphi_R |u(t)|^2 dx + 4 \int_{|x|>R} \frac{\varphi'_R(r)}{r} |\nabla u(t)|^2 dx \\
		&\mathrel{\phantom{=}} + 4 \int_{|x|>R} \left( \frac{\varphi''_R(r)}{r^2} - \frac{\varphi'_R(r)}{r^3}\right) |x\cdot \nabla u(t)|^2 dx - 2 \int_{|x|>R} \frac{\varphi'_R(r)}{r} x \cdot \nabla V |u(t)|^2 dx \\ 
		&\mathrel{\phantom{=}} -\frac{2\alpha}{\alpha+2} \int_{|x|>R} |x|^{-b} \Delta \varphi_R |u(t)|^{\alpha+2} dx - \frac{4b}{\alpha+2} \int_{|x|>R} |x|^{-b} \frac{\varphi'_R(r)}{r} |u(t)|^{\alpha+2} dx.
		\end{align*}
		By the Cauchy-Schwarz inequality $|x\cdot \nabla u| \leq |x| |\nabla u| = r |\nabla u|$ and the fact $\varphi''_R(r) \leq 2$, we see that
		\begin{align*}
		4 \int_{|x|>R} &\frac{\varphi'_R(r)}{r} |\nabla u(t)|^2 dx + 4 \int_{|x|>R} \left( \frac{\varphi''_R(r)}{r^2} - \frac{\varphi'_R(r)}{r^3}\right) |x\cdot \nabla u(t)|^2 dx - 8 \|\nabla u(t)\|^2_{L^2(|x|>R)} \\
		&\leq  4\int_{|x|>R} \left(\frac{\varphi'_R(r)}{r} -2 \right) |\nabla u(t)|^2 dx + 4 \int_{|x|>R} \frac{1}{r^2}\left( 2 - \frac{\varphi'_R(r)}{r}\right) |x\cdot \nabla u(t)|^2 dx \leq 0.
		\end{align*}
		Using the fact $\frac{\varphi'_R(r)}{r} \leq 2$, we see that
		\begin{align*}
		\left|4\int_{|x|>R} x\cdot \nabla V |u(t)|^2 dx - 2 \int_{|x|>R}\frac{\varphi'_R(r)}{r} x \cdot \nabla V |u(t)|^2 dx \right| &\leq 8\int_{|x|>R} |x\cdot \nabla V| |u(t)|^2 dx \\
		&\leq 8 \|x\cdot \nabla V\|_{L^{\frac{3}{2}}(|x|>R)} \|u(t)\|^2_{L^6} \\
		&\lesssim \|x\cdot \nabla V\|_{L^{\frac{3}{2}}(|x|>R)} \|u(t)\|^2_{H^1} \\
		&= o_R(1) \|u(t)\|^2_{H^1}.
		\end{align*}
		Moreover,
		\begin{align*}
		\frac{4(3\alpha+2b)}{\alpha+2} &\int_{|x|>R} |x|^{-b} |u(t)|^{\alpha+2} dx -\frac{2\alpha}{\alpha+2} \int_{|x|>R} |x|^{-b} \Delta \varphi_R |u(t)|^{\alpha+2} dx \\
		&\mathrel{\phantom{\int_{|x|>R} |x|^{-b} |u(t)|^{\alpha+2} dx}}- \frac{4b}{\alpha+2} \int_{|x|>R} |x|^{-b} \frac{\varphi'_R(r)}{r} |u(t)|^{\alpha+2} dx  \\
		&= \frac{2\alpha}{\alpha+2} \int_{|x|>R} |x|^{-b} (6-\Delta\varphi_R) |u(t)|^{\alpha+2} dx + \frac{4b}{\alpha+2} \int_{|x|>R} |x|^{-b} \left(2 -\frac{\varphi'_R(r)}{r}\right) |u(t)|^{\alpha+2} dx.
		\end{align*}
		Since $\Delta \varphi_R \leq 6$ and $\frac{\varphi'_R(r)}{r} \leq 2$, the above quantity is bounded (up to a constant) by $\mathlarger{\int}_{|x|>R} |x|^{-b} |u(t)|^{\alpha+2} dx$ which is bounded by
		\[
		R^{-b} \|u(t)\|^{\alpha+2}_{L^{\alpha+2}(|x|>R)} \lesssim R^{-b} \|u(t)\|^{\frac{4-\alpha}{2}}_{L^2(|x|>R)} \|u(t)\|^{\frac{3\alpha}{2}}_{L^6(|x|>R)} \lesssim R^{-b} \|u(t)\|^{\frac{4-\alpha}{2}}_{L^2(|x|>R)} \|u(t)\|^{\frac{3\alpha}{2}}_{H^1}.
		\]
		Collecting the above estimates, we end the proof.
	\end{proof}

	We are now able to prove Proposition $\ref{prop-blow-crite}$.
	
	\noindent {\bf Proof of Proposition $\ref{prop-blow-crite}$.} If $T^*<+\infty$, then we are done. If $T^*=+\infty$, then assume by contradiction that 
	\begin{align} \label{solu-boun-blow}
	\sup_{t\in [0,+\infty)} \|\nabla u(t)\|_{L^2}<\infty.
	\end{align}
	By Lemma $\ref{lem-L2-outside-ball}$, there exists $C>0$ such that for any $\vareps>0$ and any $R>0$,
	\begin{align} \label{L2-outside-ball}
	\|u(t)\|^2_{L^2(|x|>R)} \leq o_R(1) + \vareps
	\end{align}
	for all $t\in [0, T]$ with $T:= \frac{\vareps R}{C}$. By Lemma $\ref{lem-local-viri-est}$, \eqref{solu-boun-blow} and \eqref{L2-outside-ball}, we see that
	\[
	\frac{d^2}{dt^2} V_{\varphi_R}(t) \leq 8K(u(t)) + o_R(1) + CR^{-b} (o_R(1) + \vareps)^\frac{4-\alpha}{4}
	\]
	for all $t\in [0,T]$. By choosing $\vareps>0$ small enough and $R>0$ large enough so that
	\[
	o_R(1) + CR^{-b} \left(o_R(1)+\vareps\right)^{\frac{4-\alpha}{4}} \leq 4\delta,
	\]
	we get from \eqref{ass-K} that
	\[
	\frac{d^2}{dt^2} V_{\varphi_R}(t) \leq -8 \delta + o_R(1) + CR^{-b} (o_R(1) +\vareps)^{\frac{4-\alpha}{4}} \leq -4\delta
	\]
	for all $t\in [0,T]$. It follows that
	\[
	V_{\varphi_R}(T) \leq V_{\varphi_R}(0) + V'_{\varphi_R}(0) T - 2\delta T^2 \leq V_{\varphi_R}(0) + V'_{\varphi_R}(0) \frac{\vareps R}{C} - 2 \delta \frac{\vareps^2 R^2}{C^2}.
	\]
	We also have from \cite[(29)]{DWZ} that
	\[
	V_{\varphi_R}(0) =o_R(1) R^2, \quad V'_{\varphi_R}(0) = o_R(1) R.
	\]
	We thus get
	\[
	V_{\varphi_R}(T) \leq (o_R(1) - 2\tilde{\delta}) R^2,
	\]
	where $\tilde{\delta}:= \frac{\delta \vareps^2}{C^2}>0$. Taking $R>0$ large enough, we obtain $V_{\varphi_R}(T) \leq -\tilde{\delta} R^2<0$ which is a contradiction. The proof is complete.	
	\hfill $\Box$

	\section*{Acknowledgement}
	This work was supported in part by the Labex CEMPI (ANR-11-LABX-0007-01). The author would like to express his deep gratitude to his wife - Uyen Cong for her encouragement and support. He also would like to thank the reviewer for his/her helpful comments and suggestions. 
	
	\appendix
	
	\section{Energy scattering for the defocusing NLS}
	In this section, we give the proof of the energy scattering for non-radial $H^1$ solutions to the defocusing problem \eqref{NLS-V}. The proof is based on the interaction Morawetz inequality. A similar result for the repulsive inverse-power potentials was established in \cite{Dinh-repul}. Let us start with the following classical Morawetz inequality.
	\begin{lemma} \label{lem-clas-mora-est}
		Let $\frac{4}{3}<\alpha<4$. Let $V: \R^3 \rightarrow \R$ be radially symmetric satisfying \eqref{ass-V-1}, \eqref{ass-V-2}, $x\cdot \nabla V \leq 0$ and $\partial_r V \in L^q$ for any $\frac{3}{2}\leq q \leq \infty$. Let $u:I \times \R^3 \rightarrow \C$ be a $H^1$ solution to the defocusing problem \eqref{NLS-V}. Then it holds that
		\[
		-\int_I \int_{\R^3} \partial_r V |u(t)|^2 dxdt + \int_I \int_{\R^3} |x|^{-1} |u(t)|^{\alpha+2} dx dt \lesssim \sup_{t\in I} |M_{|x|}(t)| \lesssim \|u\|_{L^\infty(I,L^2)} \|\nabla u\|_{L^\infty(I,L^2)}.
		\]
	\end{lemma}

	\begin{proof}
	We first note that since $V$ is radially symmetric, the condition $x\cdot \nabla V \leq 0$ is equivalent to $\partial_r V \leq 0$. By Lemma $\ref{lem-virial-iden}$ with $b=0$, we have for a radial function $\varphi$,
	\begin{align*}
	\frac{d}{dt} M_\varphi(t)= &-\int \Delta^2 \varphi |u(t)|^2 dx + 4\sum_{j,k=1}^3 \int \partial^2_{jk} \varphi \rea (\partial_j \overline{u}(t) \partial_k u(t)) dx \\
	&- 2 \int \partial_r \varphi \partial_r V |u(t)|^2 dx + \frac{2\alpha}{\alpha+2} \int \Delta \varphi |u(t)|^{\alpha+2} dx.
	\end{align*}
	Applying the above identity to $\varphi(x) = |x|$ with the fact
	\[
	\partial_r \varphi =1, \quad \Delta \varphi = \frac{2}{|x|}, \quad \partial^2_{jk} \varphi = \frac{\delta_{jk}}{|x|} - \frac{x_j x_k}{r^3}, \quad -\Delta^2 \varphi = 8 \pi \delta_0,
	\]
	we obtain
	\begin{align*}
	\frac{d}{dt} M_{|x|}(t) = 8 \pi |u(t,0)|^2 &+ 4 \int \frac{1}{|x|} \left( |\nabla u(t)|^2 - \left|\frac{x\cdot \nabla u(t)}{|x|} \right|^2 \right) dx \\
	&- 2 \int \partial_r V |u(t)|^2 dx + \frac{4\alpha}{\alpha+2} \int |x|^{-1} |u(t)|^{\alpha+2} dx.
	\end{align*}
	Note that
	\[
	\sum_{j,k=1}^3 \partial^2_{jk}(|x|) \rea (\partial_j \overline{u} \partial_k u) = \frac{1}{|x|} \left( |\nabla u|^2 - \left|\frac{x\cdot\nabla u}{|x|}\right|^2\right) = \frac{1}{|x|} \left| \nabla u - \frac{x}{|x|} \left(\frac{x}{|x|}\cdot \nabla u\right)\right|^2 \geq 0.
	\]
	It follows that
	\[
	-\int \partial_r V |u(t)|^2 dx + \int |x|^{-1} |u(t)|^{\alpha+2} dx \lesssim \frac{d}{dt} M_{|x|}(t)
	\]
	and the result follows by taking the integration in time.	
	\end{proof}
	
	\begin{remark} \label{rem-Lq}
		The condition $\partial_r V \in L^q$ for any $\frac{3}{2}\leq q \leq \infty$ is needed to ensure $\partial_r V |u(t)|^2 \in L^1$. In fact, 
		\begin{align*}
		\left| \int \partial_r V |u(t)|^2 dx \right| \leq \|\partial_r V\|_{L^q} \|u(t)\|^2_{L^{\frac{2q}{q-1}}} \leq C\|\partial_r V\|_{L^q} \|u(t)\|^2_{H^1},
		\end{align*}
		where we have used the Sobolev embedding $H^1(\R^3) \hookrightarrow L^{\frac{2q}{q-1}}(\R^3)$ for any $\frac{3}{2}\leq q \leq \infty$.
	\end{remark}
	
	Now let $u,v$ be solutions to 
	\[
	\left\{
	\begin{array}{rcl}
	i\partial_t u + \Delta_x u &=& N(u), \quad (t,x) \in \R \times \R^m, \\
	i\partial_t v + \Delta_y v &=& N(v), \quad (t,y) \in \R \times \R^n.
	\end{array}
	\right.
	\]
	Denote
	\[
	w(t,z)= (u\otimes v)(t,z):= u(t,x) v(t,y).
	\]
	It follows that
	\begin{align} \label{equ-w}
	i\partial_t w + \Delta_z w = N(w),
	\end{align}
	where $\Delta_z:= \Delta_x + \Delta_y$ and $N(w) = N(u)v + N(v) u$. Given a real-valued function $\Psi$ on $\R^m \times \R^n$, we deine the interaction Morawetz action
	\[
	M^{\otimes 2}_\Psi (t):= 2\int \nabla_z \Psi \cdot \ima (\overline{w}(t) \nabla_z w(t)) dz.
	\]
	A direct computation gives the following result (see e.g. \cite{CGT}).
	
	\begin{lemma} [\cite{CGT}] Let $\Psi$ be a sufficiently smooth and decaying function. If $w$ is a solution to \eqref{equ-w}, then it holds that
		\begin{align*}
		\frac{d}{dt} M^{\otimes 2}_\Psi(t) = &- \int (\Delta^2_x \Psi + \Delta^2_y \Psi) |u(t)|^2 |v(t)|^2 dz \\
		& +4 \sum_{j,k} \int \partial^2_{jk} \Psi \rea (\partial_j \overline{u}(t) \partial_k u(t)) |v(t)|^2 dz + 4 \sum_{j,k} \int \partial^2_{jk} \Psi \rea (\partial_j \overline{v}(t) \partial_k v(t)) |u(t)|^2 dz \\
		& + 2 \int \nabla_x \Psi \cdot \{N(u), u\}_p(t) |v(t)|^2 dz + 2 \int \nabla_y \Psi \cdot \{N(v),v\}_p(t) |u(t)|^2 dz,
		\end{align*}
		where $\{f,g\}_p = \rea(f\nabla \overline{g} - g \nabla \overline{f})$ is the momentum bracket.
	\end{lemma}
	
	\begin{corollary} \label{coro-inter-mora-iden}
		Let $0<\alpha<4$. Let $V: \R^3 \rightarrow \R$ satisfy \eqref{ass-V-1} and \eqref{ass-V-2}. Let $\Psi: \R^3 \times \R^3 \rightarrow \R$ be a sufficiently smooth and decaying function. Let $u:I \times \R^3 \rightarrow \C$ be a solution to the defocusing problem \eqref{NLS-V}. Set
		\[
		w(t,z):= u(t,x) u(t,y), \quad x,y\in \R^3.
		\]
		Then it holds that
		\begin{align*}
		\frac{d}{dt} M^{\otimes 2}_\Psi(t) = &- 2\int \Delta^2_x \Psi(x,y) |u(t,x)|^2 |u(t,y)|^2 dx dy \\
		&+ 8 \sum_{j,k} \int \partial^2_{jk} \Psi(x,y) \rea (\partial_j \overline{u}(t,x) \partial_k u(t,x)) |u(t,y)|^2 dx dy \\
		&- 4 \int \nabla_x \Psi(x,y) \cdot \nabla_x V(x) |u(t,x)|^2 |u(t,y)|^2 dx dy \\
		&+\frac{4\alpha}{\alpha+2} \int \Delta_x \Psi(x,y) |u(t,x)|^{\alpha+2} |u(t,y)|^2 dx dy.
		\end{align*}
	\end{corollary}

	\begin{proposition} \label{prop-inter-mora-est}
		Let $0<\alpha<4$. Let $V:\R^3 \rightarrow \R$ be radially symmetric satisfying \eqref{ass-V-1}, \eqref{ass-V-2}, $x\cdot \nabla V \leq 0$ and $\partial_r V \in L^q$ for any $\frac{3}{2}\leq q \leq \infty$. Let $u: I \times \R^3 \rightarrow \C$ be a $H^1$ solution to the defocusing problem \eqref{NLS-V}. Then it holds that
		\[
		\int_I \int_{\R^3} |u(t,x)|^4 dx dt \lesssim \|u\|^3_{L^\infty(I,L^2)} \|\nabla u\|_{L^\infty(I,L^2)}.
		\]
	\end{proposition}
	
	\begin{proof}
		Let $\Psi(x,y) = |x-y|$. A direct computation shows
		\[
		\nabla_x \Psi = \frac{x-y}{|x-y|}, \quad \Delta_x \Psi = \frac{2}{|x-y|}, \quad \partial^2_{jk} \Psi = \frac{\delta_{jk}}{|x-y|} - \frac{(x_j-y_j)(x_k -y_k)}{|x-y|^3}, \quad -\Delta^2_x \Psi = 8 \pi \delta_{x=y}.
		\]
		As above, we note that
		\[
		\sum_{j,k=1}^3 \partial^2_{jk}\Psi \rea (\partial_j \overline{u} \partial_k u) = \frac{1}{|x-y|} \left|\nabla u - \frac{x-y}{|x-y|} \left(\frac{x-y}{|x-y|} \cdot \nabla u \right) \right|^2  \geq 0.
		\]
		Applying Corollary $\ref{coro-inter-mora-iden}$ and dropping positive terms, we get
		\begin{align*}
		\frac{d}{dt} M^{\otimes 2}_{|x-y|} (t) \geq 16 \pi \int |u(t,x)|^4 dx &- 4 \int \frac{(x-y)\cdot x}{|x-y||x|} \partial_r V |u(t,x)|^2 |u(t,y)|^2 dx dy \\
		&+ \frac{8\alpha}{\alpha+2} \int \frac{(x-y)\cdot x}{|x-y||x|} |x|^{-1} |u(t,x)|^{\alpha+2} |u(t,y)|^2 dx dy.
		\end{align*}
		It follows that
		\[
		\int |u(t,x)|^4 dx \lesssim \frac{d}{dt} M^{\otimes 2}_{|x-y|}(t) - \int \partial_r V |u(t,x)|^2 |u(t,y)|^2 dx dy + \int |x|^{-1} |u(t,x)|^{\alpha+2} |u(t,y)|^2 dxdy.
		\]
		Taking the integration in time, we obtain
		\begin{align*}
		\int_I \int_{\R^3} |u(t,x)|^4 dx dt \lesssim \sup_{t\in I} |M^{\otimes 2}_{|x-y|}(t)| &- \int_I \int_{\R^3\times \R^3} \partial_r V |u(t,x)|^2 |u(t,y)|^2 dx dy \\
		&+ \int_I \int_{\R^3\times \R^3} |x|^{-1} |u(t,x)|^{\alpha+2} |u(t,y)|^2 dx dy.
		\end{align*}
		We infer from Lemma $\ref{lem-clas-mora-est}$ that
		\[
		\int_I \int_{\R^3} |u(t,x)|^4 dx dt \lesssim \sup_{t\in I} |M^{\otimes 2}_{|x-y|}(t)| + \left(\sup_{t\in I} |M_{|x|}(t)| \right) \|u\|^2_{L^\infty(I,L^2)} \lesssim \|u\|^3_{L^\infty(I,L^2)} \|\nabla u\|_{L^\infty(I,L^2)}.
		\]
		The proof is complete.
	\end{proof}

	A direct consequence of Proposition $\ref{prop-inter-mora-est}$ and the conservation of mass and energy is the following result.
	\begin{corollary} \label{coro-inter-mora-est}
		Let $0<\alpha<4$. Let $V: \R^3 \rightarrow \R$ be radially symmetric satisfying \eqref{ass-V-1}, \eqref{ass-V-2}, $x\cdot \nabla V \leq 0$ and $\partial_r V \in L^q$ for any $\frac{3}{2}\leq q \leq \infty$. Let $u_0\in H^1$ and $u$ be the corresponding global solution to the defocusing problem \eqref{NLS-V}. Then it holds that
		\begin{align} \label{glo-mora-bound}
		\|u\|_{L^4(\R \times \R^3)} \leq C(E,M)<\infty.
		\end{align}
	\end{corollary}

	To show the energy scattering, we need the following nonlinear estimates.
	\begin{lemma} \label{lem-non-est-defocus}
		Let $\frac{4}{3} <\alpha<4$. Then there exist $\epsilon>0$ small enough and $\theta_1=\theta_1(\epsilon), \theta_2=\theta_2(\epsilon) \in (0,1)$ such that for any time interval $I \subset \R$,
		\[
		\|\scal{\nabla} (|u|^\alpha u) \|_{L^2(I,L^{\frac{6}{5}})} \lesssim \|\scal{\nabla} u\|^{1+\alpha(1-\theta_2)}_{S(L^2,I)} \|u\|^{\alpha \theta_1\theta_2}_{L^4(I\times \R^3)} \|u\|^{\alpha(1-\theta_1) \theta_2}_{L^\infty(I,H^1)},
		\]
		where $S(L^2,I)$ is as in \eqref{str-cha-norm}.
	\end{lemma}
	
	\begin{proof}
		We estimate
		\[
		\|\scal{\nabla} (|u|^\alpha u)\|_{L^2(I,L^{\frac{6}{5}})} \lesssim \|\scal{\nabla} u\|_{L^{4+\epsilon}(I,L^{\frac{6(4+\epsilon)}{8+3\epsilon}})} \|u\|^\alpha_{L^{\frac{2\alpha(4+\epsilon)}{2+\epsilon}}(I,L^{\frac{3\alpha(4+\epsilon)}{6+\epsilon}})}.
		\]
		It is easy to see that $\left(4+\epsilon,\frac{6(4+\epsilon)}{8+3\epsilon}\right)$ is a Schr\"odinger admissible pair and the condition $2 \leq \frac{6(4+\epsilon)}{8+3\epsilon}<3$ is satisfied for any $\epsilon>0$.  It follows that the first factor in the right hand side is bounded  by $\|\scal{\nabla} u\|_{S(L^2,I)}$. We now denote
		\[
		(q,r) := \left(\frac{2\alpha(4+\epsilon)}{2+\epsilon}, \frac{3\alpha(4+\epsilon)}{6+\epsilon}\right).
		\]
		We see that $(q,r) \in \Lambda_{\betc}$ with $\betc :=\frac{3}{2}-\frac{2}{\alpha}$, where $(q,r) \in \Lambda_\beta$ with $0<\beta \leq 1$ means 
		\[
		\frac{2}{q}+\frac{3}{r} = \frac{3}{2} - \beta.
		\]
		Since we are considering $\alpha \in \left(\frac{4}{3},4\right)$, we have $\betc \in (0,1)$.
		\begin{center}
			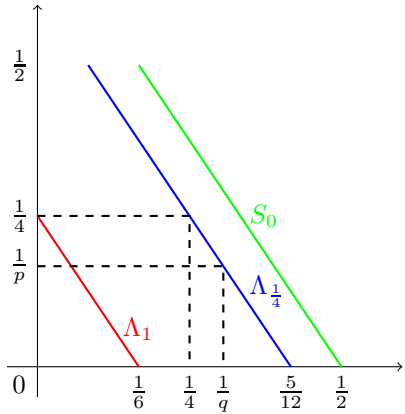
\begin{figure}[ht] \label{fig:adm-pai-3d}
				\begin{tikzpicture} [scale=8]
				\draw [->] (-0.05,0) -- (0.6,0);
				\draw [->] (0,-0.05) -- (0, 0.6);
				\draw (-0.03, -0.03) node {$0$};
				\draw (1/6, 0) node[below] {$\frac{1}{6}$};
				\draw (1/4, 0) node[below] {$\frac{1}{4}$};
				\draw (11/36, 0) node[below] {$\frac{1}{q}$};
				\draw (5/12, 0) node[below] {$\frac{5}{12}$};
				\draw (1/2, 0) node[below] {$\frac{1}{2}$};
				\draw (0, 1/6) node[left] {$\frac{1}{p}$};
				\draw (0, 1/4) node[left] {$\frac{1}{4}$};
				\draw (0, 1/2) node[left] {$\frac{1}{2}$};
				\draw [thick, green] (1/6,1/2) -- (1/2,0);
				\draw [green] (1/3, 1/4) node[right] {$S_0$};
				\draw [thick, blue] (1/12,1/2) -- (5/12,0);
				\draw[blue] (1/3, 1/8) node[right] {$\Lambda_{\frac{1}{4}}$};
				\draw [thick, red] (0,1/4) -- (1/6,0);
				\draw[red] (1/8, 1/16) node[right] {$\Lambda_1$};
				\draw [thick, dashed] (0, 1/4) -- (1/4, 1/4)-- (1/4,0);
				\draw [thick, dashed] (0, 1/6) -- (11/36, 1/6)-- (11/36, 0);
				\end{tikzpicture}
				\caption{Admissible pairs in 3D}
			\end{figure}
		\end{center}
		\noindent {\bf Case 1: $\betc= \frac{1}{4}$ or $\alpha =\frac{8}{5}$.} We use H\"older's inequality to have
		\[
		\|u\|_{L^q(I,L^r)} \leq \|u\|_{L^4(I,L^4)}^{\theta_1} \|u\|^{1-\theta_1}_{L^\infty(I, L^{\frac{12}{5}})} \lesssim \|u\|_{L^4(I,L^4)}^{\theta_1} \|u\|^{1-\theta_1}_{L^\infty(I, H^1)}. 
		\]
		To make the above estimates valid, we need to check that $\theta_1 \in (0,1)$. Note that $\theta_1=\frac{4}{q} = \frac{5(2+\eps)}{4(4+\eps)}$. By choosing $\eps>0$ small enough, the condition $\theta_1 \in (0,1)$ is fulfilled. 
		
		\noindent {\bf Case 2: $\betc \in \left(0, \frac{1}{4}\right)$ or $\frac{4}{3} <\alpha <\frac{8}{5}$.} In this case, there exist $r_1, r_2$ such that $r_1<r<r_2$ and
		\begin{align} \label{pro-1}
		(q,r_1) \in S_0, \quad (q,r_2) \in \Lambda_{\frac{1}{4}}.
		\end{align}
		We thus obtain for some $\theta_1, \theta_2 \in (0,1)$ that
		\begin{align} \label{pro-2}
		\begin{aligned}
		\|u\|_{L^q(I,L^r)} &\leq \|u\|^{1-\theta_2}_{L^q(I,L^{r_1})} \|u\|^{\theta_2}_{L^q(I,L^{r_2})} \\
		&\leq \|u\|^{1-\theta_2}_{L^q(I,L^{r_1})} \left(\|u\|^{\theta_1}_{L^4(I,L^4)} \|u\|^{1-\theta_1}_{L^\infty(I,L^{\frac{12}{5}})} \right)^{\theta_2} \\
		&\lesssim \|\scal{\nabla} u\|^{1-\theta_2}_{S(L^2,I)} \|u\|^{\theta_1\theta_2}_{L^4(I,L^4)} \|u\|^{(1-\theta_1)\theta_2}_{L^\infty(I,H^1)}.
		\end{aligned}
		\end{align}
		The above estimates are valid provided that 
		\begin{align} \label{con}
		\theta_1,\theta_2 \in (0,1), \quad 2\leq r_1 <3.
		\end{align}
		Let us check \eqref{con}. By \eqref{pro-1} and \eqref{pro-2}, we see that $\theta_2 = 4\betc \in (0,1)$, $\theta_1 = \frac{4}{q} = \frac{2(2+\eps)}{\alpha(4+\eps)}$ and $r_1 = \frac{6\alpha(4+\eps)}{12\alpha-4+(3\alpha-2)\eps}$. By taking $\eps>0$ small enough and using the fact $\frac{4}{3}<\alpha<\frac{8}{5}$, we see that $\theta_1 \in (0,1)$ and $2\leq r_1<3$.
		
		\noindent {\bf Case 3: $\betc \in \left(\frac{1}{4}, 1\right)$ or $\frac{8}{5} <\alpha<4$.} There exist $r_1, r_2$ such that $r_1<r<r_2$ and
		\[
		(q,r_1) \in \Lambda_{\frac{1}{4}}, \quad (q,r_2) \in \Lambda_1.
		\] 
		By H\"older's inequality and Sobolev embedding, we obtain for some $\theta_1, \theta_2 \in (0,1)$ that
		\begin{align*}
		\|u\|_{L^q(I,L^r)} &\leq \|u\|^{\theta_2}_{L^q(I,L^{r_1})} \|u\|^{1-\theta_2}_{L^q(I,L^{r_2})} \\
		&\lesssim \left(\|u\|^{\theta_1}_{L^4(I,L^4)} \|u\|^{1-\theta_1}_{L^\infty(I,L^{\frac{12}{5}})} \right)^{\theta_2} \|\scal{\nabla} u\|^{1-\theta_2}_{L^q(I,L^{r_3})} \\
		&\lesssim \|u\|^{\theta_1\theta_2}_{L^4(I,L^4)} \|u\|^{(1-\theta_1)\theta_2}_{L^\infty(I,H^1)} \|\scal{\nabla}u\|^{1-\theta_2}_{S(L^2,I)},
		\end{align*}
		where $(q,r_3) \in \Lambda_0$ with $\frac{1}{r_2}=\frac{1}{r_3}-\frac{1}{3}$. The above estimates hold true provided that
		\[
		\theta_1, \theta_2 \in (0,1), \quad 2\leq r_3 <3.
		\]
		We see that $\theta_2 = \frac{4}{3}(1-\betc) \in (0,1)$, $\theta_1 = \frac{4}{q} = \frac{2(2+\eps)}{\alpha(4+\eps)}$ and $r_3 = \frac{6\alpha(4+\eps)}{12\alpha-4+(3\alpha-2)\eps}$. Arguing as in Case 2, we see that the above conditions are satisfied for $\eps>0$ small enough. The proof is complete.
	\end{proof}
	
	We are now able to prove the energy scattering given in Theorem $\ref{theo-scat-defocus-NLS}$.
	
	\noindent {\bf Proof of Theorem $\ref{theo-scat-defocus-NLS}$.}
	We first show that the global Morawetz bound \eqref{glo-mora-bound} implies the global Strichartz bound
	\begin{align} \label{glo-str-bou}
	\|\scal{\nabla} u\|_{S(L^2,\R)} \leq C(E,M) <\infty.
	\end{align}
	To see this, we decompose $\R$ into a finite number of disjoint intervals $I_k=[t_k, t_{k+1}], k=1, \cdots, N$ so that
	\begin{align} \label{spl-del}
	\|u\|_{L^4(I_k\times \R^3)} \leq \delta, \quad k=1, \cdots, N,
	\end{align}
	for some small constant $\delta>0$ to be chosen later. By Strichartz estimates and the equivalence $\|\cdot\|_{W^{1,r}_V} \sim \|\cdot\|_{W^{1,r}}$, we have that
	\begin{align*}
	\|\scal{\nabla} u\|_{S(L^2,I_k)} \lesssim \|u(t_k)\|_{H^1} + \|\scal{\nabla} (|u|^\alpha u)\|_{L^2(I_k, L^{\frac{6}{5}})}. 
	\end{align*}
	We have from Lemma $\ref{lem-non-est-defocus}$ that
	\[
	\|\scal{\nabla} (|u|^\alpha u) \|_{L^2(I_k, L^{\frac{6}{5}})} \lesssim \|\scal{\nabla} u\|^{1+\alpha(1-\theta_2)}_{S(L^2,I_k)} \|u\|^{\alpha \theta_1\theta_2}_{L^4(I_k,L^4)} \|u\|^{\alpha(1-\theta_1)\theta_2}_{L^\infty(I_k,H^1)},
	\]
	for some $0<\theta_1<1$ and $0<\theta_2 \leq 1$. Thus
	\begin{align} \label{ene-sca-pro-2}
	\|\scal{\nabla} u\|_{S(L^2,I_k)} \lesssim \|u(t_k)\|_{H^1} + \|u\|^{1+\alpha(1-\theta_2)}_{S(L^2,I_k)} \delta^{\alpha\theta_1\theta_2}.
	\end{align}
	Taking $\delta>0$ small enough, we get 
	\[
	\|\scal{\nabla} u\|_{S(L^2,I_k)} \lesssim \|u(t_k)\|_{H^1} \leq C(E,M)<\infty, \quad k =1, \cdots, N.
	\]
	By summing over all intervals $I_k, k=1, \cdots, N$, we obtain \eqref{glo-str-bou}. 
	
	We now show the scattering property. By the time reversal symmetry, it suffices to treat positive times. By Duhamel's formula, we have 
	\[
	e^{it\Hc} u(t) = u_0 - i \int_0^t e^{is\Hc} |u(s)|^\alpha u(s) ds.
	\]
	Let $0<t_1<t_2<+\infty$. By Strichartz estimates and the Sobolev norms equivalence,
	\[
	\|e^{it_2\Hc} u(t_2) - e^{it_1\Hc} u(t_1)\|_{H^1} = \left\|-i\int_{t_1}^{t_2} e^{is\Hc} |u(s)|^\alpha u(s) ds \right\|_{H^1} \lesssim \|\scal{\nabla}(|u|^\alpha u)\|_{L^2([t_1,t_2], L^{\frac{6}{5}})}. 
	\]
	It follows from Lemma $\ref{lem-non-est-defocus}$ that
	\[
	\|\scal{\nabla} (|u|^\alpha u) \|_{L^2([t_1,t_2], L^{\frac{6}{5}})} \lesssim \|\scal{\nabla} u\|^{1+\alpha(1-\theta_2)}_{S(L^2,[t_1,t_2])} \|u\|^{\alpha \theta_1\theta_2}_{L^4([t_1,t_2],L^4)} \|u\|^{\alpha(1-\theta_1)\theta_2}_{L^\infty([t_1,t_2],H^1)},
	\]
	for some $0<\theta_1<1$ and $0<\theta_2 \leq 1$. Thus, by \eqref{glo-mora-bound}, \eqref{glo-str-bou} and the conservation of mass and energy, we see that 
	\[
	\|e^{it_2\Hc} u(t_2) - e^{it_1\Hc} u(t_1)\|_{H^1} \rightarrow 0 \text{ as } t_1, t_2 \rightarrow +\infty.
	\]
	Hence the limit
	\[
	u_0^+:= \lim_{t\rightarrow +\infty} e^{it\Hc} u(t) = u_0 - i \int_0^{+\infty} e^{is\Hc} |u(s)|^\alpha u(s) ds
	\]
	exists in $H^1$. Moreover,
	\[
	u(t) - e^{-it\Hc} u(t) = i \int_t^{+\infty} e^{-i(t-s)\Hc} |u(s)|^\alpha u(s) ds.
	\]
	Minicing the above estimates, we prove as well that
	\[
	\|u(t)-e^{-it\Hc} u(t)\|_{H^1} \rightarrow 0 \text{ as } t \rightarrow +\infty.
	\]
	The proof is complete.
	\hfill $\Box$

\end{document}